\documentclass[11pt]{article}

\usepackage{amsmath}
\usepackage{amsthm}
\usepackage{amsfonts}
\usepackage{amssymb}
\usepackage{enumerate}
\usepackage{hyperref}
\usepackage[bottom,symbol]{footmisc}
\usepackage{tikz}
\usetikzlibrary{calc}

\theoremstyle{plain}
\newtheorem{theorem}{Theorem}
\newtheorem{lemma}[theorem]{Lemma}
\newtheorem{corollary}[theorem]{Corollary}
\newtheorem{proposition}[theorem]{Proposition}

\newtheorem*{quotedtheorem}{Theorem}
\newtheorem{construction}{Construction}

\theoremstyle{definition}
\newtheorem{definition}{Definition}
\newtheorem{conjecture}{Conjecture}
\newtheorem{question}[conjecture]{Question}
\newtheorem{problem}{Problem}

\theoremstyle{remark}
\newtheorem{remark}[theorem]{Remark}

\newcommand{\e}{\varepsilon}
\newcommand{\B}[1]{{\bf #1}}
\newcommand{\C}[1]{{\protect\cal #1}}
\newcommand{\I}[1]{{\mathbb #1}}

\newcommand{\pp}{P}
\newcommand{\jsection}[1]{part \texttt{"#1"}}
\newcommand{\jSection}[1]{Part \texttt{"#1"}}
\newcommand{\flag}[1]{\mbox{\texttt{"#1"}}}
\newcommand{\CT}{\C T}

\title{THE CODEGREE THRESHOLD FOR 3-GRAPHS WITH INDEPENDENT NEIGHBOURHOODS}
\author{V. Falgas--Ravry\thanks{Department of Mathematics, Vanderbilt University, Nashville TN 37240, USA, and Institutionen f\"or matematik och matematisk statistik, Ume{\aa}  Universitet, 901 87 Ume{\aa}, Sweden. Supported by the Kempe foundation.} \and
E. Marchant\thanks{29 Woodside Close, HP6 5EF Amersham, UK. Research funded by
Trinity College, Cambridge.} \and 
O. Pikhurko\thanks{Mathematics Institute and DIMAP, University of Warwick, CV4 7AL Coventry, UK. Supported by ERC grant~306493 and EPSRC grant~EP/K012045/1.} \and E.~R. Vaughan\thanks{Centre for Discrete Mathematics, Queen Mary University of London, E1 4NS London, UK.}}

\begin{document}
\maketitle
\enlargethispage{4mm}
\begin{abstract}
Given a family of $3$-graphs $\mathcal{F}$, we define its codegree threshold
$\mathrm{coex}(n, \mathcal{F})$ to be the largest number $d=d(n)$ such that there exists an $n$-vertex $3$-graph in which every pair of vertices is contained in at least $d$  $3$-edges but which contains no member of $\mathcal{F}$ as a subgraph. Let $F_{3,2}$ be the $3$-graph on $\{a,b,c,d,e\}$ with $3$-edges
$abc$, $abd$, $abe$ and $cde$.

In this paper, we give two proofs that 
\[\mathrm{coex}(n, \{F_{3,2}\})= \left( \frac{1}{3}+o(1)\right)n,\]
the first by a direct combinatorial argument and the second via a flag algebra
computation. Information extracted from the latter proof is then used to obtain
a stability result, from which in turn we derive the exact codegree threshold
for all sufficiently large $n$:
\[\mathrm{coex}(n, \{F_{3,2}\})= \left\{ \begin{array}{ll}
\lfloor n/3 \rfloor-1 & \textrm{if $n$ is congruent to $1$ modulo $3$,}\\
\lfloor n/3 \rfloor & \textrm{otherwise}.
\end{array} 
\right.\]
In addition we determine the set of codegree-extremal configurations for all sufficiently large $n$.
\end{abstract}


\pagestyle{myheadings}
\thispagestyle{plain}
\markboth{V. FALGAS-RAVRY, E. MARCHANT, O. PIKHURKO, E.~R. VAUGHAN}{THE CODEGREE THRESHOLD OF $F_{3,2}$}

\section{Introduction}
\subsection{Tur\'an-type problems}\label{turansection}
We begin with some standard definitions. Let $r,n \in \mathbb{N}$. We write
$[n]$ for the discrete interval $\{1,2, \ldots n\}$. Also, given a set $S$ we
denote by $S^{(r)}$ the collection of all $r$-subsets from $S$.

An \emph{$r$-graph} is a pair of sets $G=(V,E)$, where $V=V(G)$ is a set of
\emph{vertices} and $E=E(G)$ is a collection of $r$-sets from $V$, which
constitute the \emph{$r$-edges} of $G$. An $r$-graph $G$ is \emph{nonempty} if
$E(G) \neq \emptyset$. A \emph{subgraph} of $G$ is an $r$-graph $H$ with $V(H)
\subseteq V(G)$ and $E(H) \subseteq E(G)$. Given a family of $r$-graphs
$\mathcal{F}$, we say that $G$ is \emph{$\mathcal{F}$-free} if no
member of $\mathcal{F}$ is isomorphic to a subgraph of $G$.

One of the central problems in extremal combinatorics is determining the maximum
number $\textrm{ex}(n, \mathcal{F})$ of $r$-edges that an $r$-graph on $n$ vertices
may contain while remaining $\mathcal{F}$-free, where $\mathcal{F}$ is a family
of nonempty $r$-graphs. The function $n \mapsto \textrm{ex}(n, \mathcal{F})$ is
known as the \emph{Tur\'an number} of $\mathcal{F}$.
\begin{problem}\label{Turannumber}
Let $\mathcal{F}$ be a family of nonempty $r$-graphs. Determine the Tur\'an
number of $\mathcal{F}$.
\end{problem}
Often computing the Tur\'an number exactly may be difficult, and so, lowering
our sights, we are interested in the asymptotic behaviour of the Tur\'an
function: what is the asymptotically maximal proportion of all possible edges
that an $\mathcal{F}$-free $r$-graph may contain? An easy averaging argument
shows that the nonnegative sequence $\textrm{ex}(n, \mathcal{F})/ \binom{n}{r}$
is nonincreasing, and hence converges to a limit as $n$ tends to infinity. This
limit is known as the \emph{Tur\'an density} of $\mathcal{F}$, and denoted by
$\pi(\mathcal{F})$.
\begin{problem}\label{Turandensity}
Let $\mathcal{F}$ be a family of nonempty $r$-graphs. Determine the Tur\'an
density of $\mathcal{F}$.
\end{problem}

These two problems have been studied very successfully in the case $r=2$,
corresponding to ordinary ($2$-)graphs. Tur\'an determined the Tur\'an number of
complete graphs~\cite{Turan41}, while Erd\H{o}s and Stone~\cite{ErdosStone46}
fully resolved Problem~\ref{Turandensity} in a seminal result relating the
Tur\'an density of a family of graphs to its chromatic number.

Despite recent progress, this stands in some contrast to the situation when
$r\geq 3$. Indeed few Tur\'an densities are known even for $3$-graphs, and the
problem of determining them is known to be hard in general. Let us introduce
here a few of the $3$-graphs relevant to our discussion. As a convention, we
will write $xyz$ for the $3$-edge $\{x,y, z\}$ and $\pi(F_1, F_2, \ldots F_t)$
for the Tur\'an density $\pi(\{F_1, F_2, \ldots F_t\})$.

Let $K_4$ denote the complete $3$-graph on $4$ vertices, and let $K_4^-$ denote
the $3$-graph obtained from $K_4$ by deleting one of its edges. Let $F_{3,2}$ be
the $3$-graph $([5], \{123, 124, 125, 345\})$. Finally, let $F_7$ be the
\emph{Fano plane}, namely the (unique up to isomorphism) $3$-graph on $7$
vertices in which every pair of vertices is contained in exactly one $3$-edge.

Almost no Tur\'an densities or Tur\'an numbers for $3$-graphs were known until
de Caen and F\"uredi~\cite{deCaenFuredi00} established that $\pi(F_7)=3/4$. (A
notable exception is a result of Bollob\'as~\cite{Bollobas74}.) The Tur\'an
number of the Fano plane was independently determined shortly afterwards by
Keevash and Sudakov~\cite{KeevashSudakov05} and F\"uredi and
Simonovits~\cite{FurediSimonovits05}. Around the same time, F\"uredi, Pikhurko
and Simonovits determined first the Tur\'an
density~\cite{FurediPikhurkoSimonovits03} and then the Tur\'an
number~\cite{FurediPikhurkoSimonovits05} of $F_{3,2}$.

The next major development as far as computing Tur\'an densities is concerned
was the advent of Razborov's semi-definite
method~\cite{Razborov10}. With the assistance of computers, this method has been
used in recent years to significantly increase the number of known Tur\'an
densities for $3$-graphs~\cite{BaberTalbot12, FalgasRavryVaughan13}.

\subsection{The codegree problem}\label{codegreesection}
Given a $3$-graph $G$ and a vertex $x \in V(G)$, the \emph{degree} $d(x)$ of $x$
in $G$ is the number of $3$-edges of $G$ containing $x$. The \emph{minimum
degree} of $G$ is $\delta(G)= \min_{x \in V(G)}d(x)$. It is not hard to see that
the Tur\'an density problem for $3$-graphs is equivalent to determining asymptotically
what minimum degree condition forces a $3$-graph on $n$ vertices to
contain a copy of a member of a given family $\mathcal{F}$ as a subgraph.

A natural variant is to consider what minimum \emph{codegree} condition is
required to force an $\mathcal{F}$-subgraph. Here, the \emph{codegree} $d(x,y)$
of two distinct vertices $x,y$ in a $3$-graph $G$ is the number of $3$-edges of
$G$ which contain the pair $\{x,y\}$. (We may sometimes write this as $d_G(x,y)$
to emphasize that we are taking the codegree in $G$ and not some other
$3$-graph.) The \emph{minimum codegree} $\delta_2(G)$ of $G$ is as the name
suggests the minimum of $d(x,y)$ over all pairs of vertices from $V(G)$.

We may then define for a family of nonempty $3$-graphs $\mathcal{F}$ the
\emph{codegree threshold} $\mathrm{coex}(n, \mathcal{F})$ to be the maximum of
$\delta_2(G)$ over all $\mathcal{F}$-free $3$-graphs $G$ on $n$ vertices.  This
is the codegree analogue of the Tur\'an number.
\begin{problem}\label{codegreethreshold}
Let $\mathcal{F}$ be a family of nonempty $3$-graphs. Determine the codegree
threshold of $\mathcal{F}$.
\end{problem}
Again it may be that in general computing the codegree threshold proves
difficult, and that we would first be interested in determining the asymptotic behaviour of
$\mathrm{coex}(n, \mathcal{F})$. Following the analogy with the Tur\'an-type
problems, it is natural to consider the sequence $\mathrm{coex}(n,
\mathcal{F})/(n-2)$ or some close relative. Here however we do not in general
have monotonicity: Lo and Markstr\"om~\cite{LoMarkstrom12} showed that neither
of $\mathrm{coex}(n, K_4)/n$ and $\mathrm{coex}(n, K_4)/(n-2)$ is nonincreasing.
The limit of $\mathrm{coex}(n, \mathcal{F})/n$ does exist however, as first
shown by Mubayi and Zhao~\cite{MubayiZhao07}. Thus we may define the
\emph{codegree density} of $\mathcal{F}$ to be
\[ \gamma(\mathcal{F}):= \lim_{n \rightarrow \infty} \frac{\mathrm{coex}(n,
\mathcal{F})}{n-2}.\]
(Obviously choosing $n$ or $n-2$ in the denominator does not affect the limit.)

This gives us a codegree analogue of the Tur\'an density for $3$-graphs.
\begin{problem}\label{codegreedensity}
Let $\mathcal{F}$ be a family of nonempty $3$-graphs. Determine the codegree
density $\gamma(\mathcal{F})$.
\end{problem}
What is the relationship between $\pi(\mathcal{F})$ and $\gamma(\mathcal{F})$?
By counting $3$-edges in two ways it is easy to show that $\gamma(\mathcal{F})\leq
\pi(\mathcal{F})$.

The first result on codegree density is due to Mubayi~\cite{Mubayi05}, who
showed $\gamma(F_7)=1/2$. This gave an example where $\gamma(\mathcal{F})$ is
strictly less than $\pi(\mathcal{F})$ (since de Caen and F\"uredi had shown
$\pi(F_7)=3/4$). The codegree threshold for the Fano plane was determined for
all sufficiently large $n$ by Keevash~\cite{Keevash09}, who used hypergraph
regularity and quasirandomness to get a stability result from which he was able
to proceed to the exact result via more standard combinatorial arguments. His
method gave slightly more than just the codegree threshold, as it also
identified exactly which $3$-graphs could attain it, namely complete bipartite
$3$-graphs. DeBiasio and Jiang~\cite{DeBiasioJiang12} later gave a simpler proof
that $\mathrm{coex}(n, \mathcal{F})= \lfloor n/2 \rfloor$ for $n$ sufficiently
large which avoided the use of regularity.

Except for the Fano plane, almost no codegree results are known for $3$-graphs.
Keevash and Zhao~\cite{KeevashZhao07} studied the codegree density of projective
geometries, following on earlier work of Keevash~\cite{Keevash05} on their
Tur\'an densities. Nagle~\cite{Nagle99} conjectured that $\gamma(K_4^-)=1/4$,
while Czygrinow and Nagle~\cite{CzygrinowNagle01} conjectured that
$\gamma(K_4)=1/2$, with lower-bound constructions coming in both cases from
random tournaments. The first author~\cite{FalgasRavry13} gave non-isomorphic lower bound constructions for $\gamma(K_t)$ for general $t$. Recently, a subset of the authors proved $\gamma(K_4^-)=1/4$
using flag algebras~\cite{FalgasRavryPikhurkoVaughan13}.

\subsection{$3$-graphs with independent neighbourhoods}
Given a $3$-graph $G$ and a pair of distinct vertices $x,y \in V(G)$, their
\emph{joint neighbourhood} in $G$ is
\[\Gamma(x,y)=\{z \in V(G): \ \{x,y,z\} \in E(G)\}.\]
In an $F_{3,2}$-free $3$-graph, the joint neighbourhoods form independent
(edge-free) subsets of the vertex set. Such $3$-graphs are thus said to have
\emph{independent neighbourhoods}.

As mentioned in Section~\ref{turansection}, the Tur\'an density and Tur\'an
number of $F_{3,2}$ were determined by F\"uredi, Pikhurko and
Simonovits~\cite{FurediPikhurkoSimonovits03,FurediPikhurkoSimonovits05}, who
showed that the extremal configurations were `one-way bipartite' $3$-graphs.
\begin{figure}
\centering

\begin{tikzpicture}

\draw (0,0) circle [radius=1.418];
\draw (3,0) circle [radius=1];
\fill[black!10] (0,-0.5) -- (3, 0) -- (0, 0.5) -- (0, -0.5);
\node (an) at (0,-1.6)[] {A};
\node (bn) at (3, -1.6)[] {B};

\end{tikzpicture}
\caption{Construction~\ref{onewaybip}}
\end{figure}
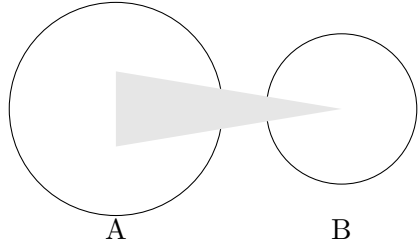

\begin{construction}\label{onewaybip}
Given a vertex set $V$ and a bipartition $V=A \sqcup B$, we define a one-way
bipartite $3$-graph $D_{A, B}$ on $V$ by taking as the $3$-edges all triples
$\{a_1,a_2, b\}$ with $a_1, a_2 \in A$ and $b \in B$.
\end{construction}
It is easy to see that $D_{A,B}$ has independent neighbourhoods, and that the
number of $3$-edges in $D_{A,B}$ is maximised when $\vert A \vert = 2 \vert B
\vert +O(1)$.
\begin{quotedtheorem}[F\"uredi, Pikhurko and
Simonovits~\cite{FurediPikhurkoSimonovits05}]
There exists $n_0\in \mathbb{N}$ such that if $G$ is a $3$-graph on $n\geq n_0$ vertices with independent neighbourhoods and
$\vert E(G)\vert =\mathrm{ex}(n, F_{3,2})$, then there exists a partition $V(G)=A\sqcup
B$ of its vertex set such that $G=D_{A,B}$. 
\end{quotedtheorem}

Bohman, Frieze, Mubayi and Pikhurko~\cite{BohmanFriezeMubayiPikhurko10}
conjectured that a natural modification of Construction~\ref{onewaybip} was
optimal for the codegree problem for $F_{3,2}$. 
\begin{construction}\label{orcyclebip}
Given a vertex set $V$, and a tripartition $V=A\sqcup B \sqcup C$, we define a
$3$-graph $T_{A,B,C}$ on $V$ by taking the union of $D_{A,B}$, $D_{B,C}$ and
$D_{C,A}$.
\end{construction}

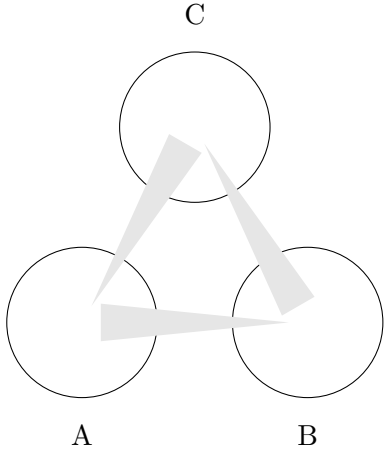
\begin{figure}
\centering
\begin{tikzpicture}

\draw (0,0) circle [radius=1];
\draw (3,0) circle [radius=1];
\coordinate (alpha) at ($ (0,0)!1!60:(3,0) $);
\draw (alpha) circle [radius=1];

\fill[black!10] (0.25,-0.25) -- (2.75, 0) -- (0.25, 0.25) -- (0.25, -0.25);

\coordinate (d) at ($ (3,0)!1!120:(3.25,-0.25) $);
\coordinate (e) at ($ (3,0)!1!120:(5.75,0) $); 
\coordinate (f) at ($ (3,0)!1!120:(3.25,0.25) $);
\fill[black!10] (d) -- (e)  -- (f) --(d);

\coordinate (a) at ($ (0,0)!1!240:(0.25,-0.25) $);
\coordinate (b) at ($ (0,0)!1!240:(2.75,0) $); 
\coordinate (c) at ($ (0,0)!1!240:(0.25,0.25) $);

\coordinate (a') at ($ (alpha)+(a)$);
\coordinate (b') at ($ (alpha)+(b)$);
\coordinate (c') at ($ (alpha)+(c)$);
\fill[black!10] (a') -- (b')  -- (c') --(a');

\node (cn) at ($(alpha)+(0,1.5)$)[] {C};
\node (an) at (0, -1.5)[] {A};
\node (bn) at (3, -1.5)[] {B};
\end{tikzpicture}
\caption{Construction~\ref{orcyclebip}}
\end{figure}

Again we have that $T_{A,B,C}$ has independent neighbourhoods, and 
\[\delta_2(T_{A,B,C})= \min\left( \vert A\vert, \vert B \vert, \vert C \vert
\right) -1 ,\]
which is maximised when the three parts $A, B, C$ are \emph{balanced} -- that
is, have sizes as equal as possible. Thus $\mathrm{coex}(n, F_{3,2})\geq \lfloor
n/3 \rfloor-1$. Bohman, Frieze, Mubayi and
Pikhurko~\cite{BohmanFriezeMubayiPikhurko10} conjectured that this
provides a tight lower-bound for the codegree density.
\begin{conjecture}[Bohman, Frieze, Mubayi and
Pikhurko~\cite{BohmanFriezeMubayiPikhurko10}]\label{bfmpconj}
\[\gamma(F_{3,2})=\frac{1}{3}.\]
\end{conjecture}

\subsection{Results and structure of the paper}
In this paper we show that 
\[\mathrm{coex}(n, \{F_{3,2}\})= \left\{ \begin{array}{ll}
\lfloor n/3 \rfloor-1 & \textrm{if $n$ is congruent to $1$ modulo $3$}\\
\lfloor n/3 \rfloor & \textrm{otherwise,}
\end{array}
\right.\]
for all $n$ sufficiently large, and determine the set of extremal configurations
(which are close to but distinct from balanced $T_{A,B,C}$ configurations in
general). This settles Conjecture~\ref{bfmpconj} in the affirmative and fully
resolves Problems~\ref{codegreethreshold} and~\ref{codegreedensity} for the
family $\mathcal{F}=\{F_{3,2}\}$ and $n$ sufficiently large.

We first give two proofs that the codegree density of $F_{3,2}$ is $1/3$. 
\begin{theorem}[Codegree density]\label{f32density}
\[\gamma(F_{3,2})=\frac{1}{3}.\]
\end{theorem}
In Section~\ref{edssection}, we give a purely combinatorial proof of
Theorem~\ref{f32density} due to Marchant, which appeared in his PhD
thesis~\cite{Marchant11}. In Section~\ref{flagsection}, we adapt the
semi-definite method of Razborov to the codegree setting to give a second proof
of Theorem~\ref{f32density}. While this second proof, a computer-assisted flag
algebra calculation, is not nearly so elegant, it gives us some information
about the structure of near-extremal $3$-graphs. This information can be used
together with a hypergraph removal lemma to prove a stability result. To state
this formally, we need to make one more definition.
\begin{definition}
Let $G$ and $H$ be $3$-graphs on vertex sets of size $n$  The \emph{edit
distance} between $G$ and $H$ is the minimum number of changes needed to make
$G$ into an isomorphic copy of $H$, where a change consists in replacing an edge
by a non-edge or vice versa.
\end{definition}
\begin{theorem}[Stability]\label{f32stability}
For all $\varepsilon>0$ there exist $\delta>0$ and $n_0\in\mathbb{N}$ such that if	
$G$ is an $F_{3,2}$-free $3$-graph on $n\geq n_0$ vertices with
\[ \delta_2(G)\geq\left(\frac{1}{3}-\delta\right)n,\]
then $G$ lies at edit distance at most $\varepsilon \binom{n}{3}$ from a balanced $T_{A,B,C}$ construction.
\end{theorem}

We use Theorem~\ref{f32stability} in Section~\ref{exactsection} to prove our
result on the codegree threshold:
\begin{theorem}[Codegree threshold]\label{f32number}
For all $n$ sufficiently large,
\[\mathrm{coex}(n, \{F_{3,2}\})= \left\{ \begin{array}{ll}
\lfloor n/3 \rfloor-1 & \textrm{if $n$ is congruent to $1$ modulo $3$}\\
\lfloor n/3 \rfloor & \textrm{otherwise.}
\end{array} \right.\]
\end{theorem}
In addition we determine the set of extremal configurations. Since this set
depends on the congruence class of $n$ modulo $3$ and in one case has a slightly
technical description, we postpone the corresponding theorems to
Section~\ref{exactsection} (Theorems~\ref{n=0mod3config}, \ref{n=2mod3config},
\ref{n=1mod3configt2} and~\ref{n=1mod3configt1}).

We end the paper with a discussion of `mixed problems': given $c$: $0\leq c \leq
1/3$, what is the asymptotically maximal $3$-edge density $\rho_c$ in
$F_{3,2}$-free $3$-graphs with codegree density at least $c$? We make a
conjecture regarding the value of $\rho_c$.

\section{Codegree density via extensions}\label{edssection}
In this section, we prove that $\gamma(F_{3,2})=1/3$. Our strategy is similar in
spirit to the one espoused by de Caen and F\"uredi~\cite{deCaenFuredi00} in
their work on the Tur\'an density of the Fano plane: we show that if
$\delta_2(G)$ is large then $G$ contains a copy either of $F_{3,2}$ or of some
`nice subgraph' $H$. In the latter case we repeat the procedure using the extra
assumption that $H$ is a subgraph of $G$: we find again either a copy of
$F_{3,2}$ or a copy of an even `nicer' subgraph, $H'$, and so on.

Our approach is based on Lemma~\ref{lem:extcodegree}, proved in the next
subsection, which establishes the existence of `nice' extensions of a subgraph
in a $3$-graph with high codegree. In Section~\ref{condcodegreesection}, we
define conditional codegree density -- loosely speaking, the codegree density
subject to the constraint of containing a particular subgraph $H$. This concept
then allows us to apply Lemma~\ref{lem:extcodegree} in a very streamlined
fashion in the final subsection to prove Theorem~\ref{f32density}.

\subsection{Extensions}
We prove here a useful lemma, which tells us that if we have a small subgraph
$H$ inside a $3$-graph $G$ which has a high minimum codegree $\delta_2(G)$, then
we can extend $H$ to a slightly larger `nice' subgraph $H'$ of $G$.

We begin with some definitions.
\begin{definition}
Let $H$ be a $3$-graph. A \emph{(simple) extension} of $H$ is a $3$-graph $H'$
with $V(H')=V(H)\cup\{z\}$ for some $z \notin V(H)$ and $E(H') \supseteq E(H)$.
We denote by $L(H'; H)$ the \emph{link graph} of the new vertex $z$,
\[L(H';H)= \{xy \in V(H)^{(2)}: \ xyz \in E(H')\}.\]
\end{definition}
\begin{definition}
A sequence of $3$-graphs $(G_n)_{n \in \mathbb{N}}$ \emph{tends to infinity} if
$\vert V(G_n)\vert\rightarrow \infty$ as $n \rightarrow \infty$. Also, given a
$3$-graph $H$, we say that a sequence $(G_n)_{n \in \mathbb{N}}$ \emph{contains
$H$} if all but finitely many of the $3$-graphs $G_n$ contain $H$ as a subgraph.
\end{definition}

Given a set $S$, write $\Delta(S)$ for the $(\vert S\vert-1)$-dimensional
simplex 
\[\Bigl\{\underline{\alpha}\in{[0,1]}^{ S }:\sum_{s\in S}\alpha_s=1\Bigr\}.\] 
If $H$ is a $3$-graph and $\underline{\alpha}\in\Delta({V(H)}^{(2)})$, then
$\underline{\alpha}$ is a weighting on the pairs of vertices of $H$. We can now
state and prove our key lemma.

\begin{lemma}\label{lem:extcodegree}
Let $H$ be a $3$-graph. Suppose $(G_n)_{n \in \mathbb{N}}$ is a sequence of
$3$-graphs tending to infinity with
\[c=\liminf_{n\rightarrow\infty}\frac{\delta_2(G_n)}{\vert V(G_n) \vert}\,,\]
and that $(G_n)_{n \in \mathbb{N}}$ contains $H$.
Then for any $\underline{\alpha}\in\Delta({V(H)}^{(2)})$, there is a simple
extension $H'$ of $H$ with 
\[\sum_{xy\in L(H';H)}\alpha_{xy}\geq c\,\] and a subsequence
$(G_{n_k})_{k \in \mathbb{N}}$ of $(G_n)_{n \in \mathbb{N}}$ such that
$(G_{n_k})_{k \in \mathbb{N}}$ contains $H'$.
\end{lemma}
\begin{proof}
Let $(G_n)=(G_n)_{n \in \mathbb{N}}$ be a $3$-graph sequence tending to infinity
with \[c=\liminf_{n\rightarrow\infty}\frac{\delta_2(G_n)}{\vert
V(G_n)\vert}\,.\] Suppose $H$ is a $3$-graph contained in $(G_n)$ and let
$\underline{\alpha}\in\Delta({V(H)}^{(2)})$.

We claim that for every $\varepsilon>0$ there exists an extension $H'$ of $H$
such that $H'$ is contained as a subgraph in infinitely many of the $3$-graphs
$G_n$ and the weaker condition \[\sum_{xy\in L(H';H)}\alpha_{xy}\geq
c-2\varepsilon\] holds. This is sufficient to prove the lemma as there are up to
isomorphism only finitely many possible simple extensions of $H$, and so one of
them must satisfy the weaker condition for all $\varepsilon>0$.

Fix $0<\varepsilon<1$ and choose $N \in \mathbb{N}$ sufficiently large such that
for $n \geq N$ all of the following hold:
\begin{enumerate}[(i)]
\item $\delta_2(G_n)/\vert V(G_n)\vert\geq c-\varepsilon$,
\item $\vert V(G_n)\vert \geq \vert V(H)\vert/\varepsilon$, and
\item $H$ is a subgraph of $G_n$.
\end{enumerate}
Consider a $3$-graph $G_n$ from our sequence with $n \geq N$. Fix a copy of $H$
within $G_n$ (we know by (iii) above that such a copy exists), and consider the
weighted sum \[s=\sum_{xy\in {V(H)}^{(2)}}\alpha_{xy}\vert\Gamma(x,y)\vert\,.\]
We have $s\geq (c-\varepsilon)\vert V(G_n) \vert$ by (i) above. Also,
\begin{align*}
s&=\sum_{z\in V(G_n)}\,\,\sum_{xy\in {V(H)}^{(2)}:\ xyz\in E(G_n)}\alpha_{xy} \\
&\leq\left({\sum_{z\in V(G_n)\backslash V(H)}\,\,\sum_{xy\in {V(H)}^{(2)}:\
xyz\in E(G_n)}\alpha_{xy}}\right)+\vert V(H) \vert \,.
\end{align*}
Hence by averaging there exists a vertex $z\notin V(H)$  such that 
\begin{align*}
\sum_{xy\in {V(H)}^{(2)}:\  xyz\in E(G^n)}\alpha_{xy}& \geq \frac{\vert V(G_n)
\vert}{\vert V(G_n) \setminus V(H) \vert } (c-\varepsilon)- \frac{\vert V(H)
\vert}{\vert V(G_n) \setminus V(H) \vert }\\
&\geq \frac{\vert{V(G_n)} \vert}{\vert V(G_n) \setminus V(H) \vert }
(c-2\varepsilon)&&\textrm{\hspace{-18mm}by (ii) above}\\
& > c- 2\varepsilon \,.
\end{align*}

Therefore the simple extension $H'$ of $H$ with vertex set $V(H)\cup \{z\}$ and
$3$-edges $E(H)\cup\{xyz:\ xy\in {V(H)}^{(2)},\, xyz\in E(G_n)\}$ satisfies our
weaker condition and is a subgraph of $G_n$. Since there are up to isomorphism
only finitely many extensions of $H$, one of them must satisfy the weaker
condition and be contained in infinitely many of the $3$-graphs in our sequence
$(G_n)_{n \in \mathbb{N}}$. This concludes the proof of our claim and with it
the proof of the lemma.
\end{proof}
We shall sometimes write $w_{\underline{\alpha}}(L(H';H))$, or simply $w(L)$,
for $\sum_{xy\in L(H';H)}\alpha_{xy}$. This quantity $w(L)$ is exactly the
total weight of the pairs picked up by the new vertex in the extension, with
respect to the weighting $\underline{\alpha}$.

\subsection{Conditional codegree density}\label{condcodegreesection}
Our arguments in the proof of Theorem~\ref{f32density} are of the form ``if $G$
contains $H$ and $\delta_2(G)$ is large then $G$ must contain a copy of a member
of $\mathcal{F}$''. It is thus natural to make the following definition.
\begin{definition}
Let $H$ be a $3$-graph, and let $\mathcal{F}$ be a family of nonempty
$3$-graphs. The \emph{conditional codegree threshold} of $\mathcal{F}$ given
$H$, denoted by $\mathrm{coex}(n, \mathcal{F} \vert H)$, is the maximum of
$\delta_2(G)$ over all $n$-vertex, $\mathcal{F}$-free $3$-graphs $G$ which
contain a copy of $H$ as a subgraph.
\end{definition}

Our aim in this subsection is to show that we can define a conditional codegree
density from this, in other words that the sequence $\mathrm{coex}(n,\mathcal{F}
\vert H)/n$ tends to a limit as $n \rightarrow \infty$. This will be very
similar to the proof that the usual codegree density is
well-defined~\cite{MubayiZhao07}.

\begin{lemma}\label{claim:takesmallerconditional}
Let $H$ be a $3$-graph and let $\varepsilon>0$. Then there exists an integer
$N=N(\varepsilon, H)$ such that for all $n, n' \in \mathbb{N}$ with $N \leq n'
\leq n$, every $3$-graph $G$ on $n$ vertices containing a copy of $H$ has a
subgraph $G'$ on $n'$ vertices also containing a copy of $H$ and satisfying
\[ \frac{\delta_2(G')}{n'}> \frac{\delta_2(G)}{n} - \varepsilon\, .\] 
\end{lemma}
(This is just saying that $G'$ has `codegree density' almost as large as $G$.)

\begin{proof}
Let $H$ be a $3$-graph on $h$ vertices, and let $\varepsilon>0$. Suppose $G$ is
a $3$-graph on $n$ vertices containing a copy of $H$. We form an $n'$-vertex
subgraph of $G$ by fixing a copy of $H$ in $G$ and extending it by adding $n'-h$
vertices selected uniformly at random from the rest of $G$. Let $G'$ denote the
resulting (random) induced subgraph of $G$. Clearly $G'$ contains a copy of $H$
and has the right order. Now let us show that -- provided $n$ and $n'$ are
sufficiently large -- $G'$ also has a good chance of having a reasonably high
minimal codegree.

Let $P_1, P_2, \ldots, P_{\binom{n'}{2}}$ be a random enumeration of the pairs of vertices from $V(G')$. Note that conditional on $P_i=xy$, the set $V(G')\setminus \left(P_i\cup V(H)\right)$ is distributed as a uniformly chosen random subset of $V(G) \setminus \left(P_i\cup V(H)\right)$ of size $n'-\vert V(H)\cup P_i \vert\geq n'-h-2$.

For each $i: \ 1 \leq i \leq \binom{n'}{2}$ and $t\in \mathbb{N}$, we have
\begin{align*}
\mathbb{P}(d_{G'}(P_i)\leq t)&\leq \sum_{xy \in V(G)^{(2)}} \mathbb{P}(P_i=xy)\mathbb{P}\Bigl(\left\vert\left(V(G')\cap \Gamma(x,y)\right)\setminus \left(P_i\cup V(H)\right)\right\vert \leq t \Big\vert P_i=xy\Bigr)\\
&\leq \mathbb{P}(X\leq t),
\end{align*}
where $X$ is the hypergeometric random variable
 \[X\sim\textrm{Hypergeometric}\left(n'-2-h,\delta_2(G)-h, n-h\right)\,.\]
 (Recall that the $\mathrm{Hypergeometric}(s, t, N)$ distribution with parameters $s,t\leq N$ is obtained as follows: fix a $t$-subset $A$ of an $N$-set. Then pick an $s$-set $B$ from the same $N$-set uniformly at random; the $\mathrm{Hypergeometric}(s, t, N)$ distribution is the distribution of the number of elements of $A$ included in $B$.)

Now, provided $n,n'$ are both sufficiently large,
\[ \mathbb{E}(X)\geq \frac{n'}{n}\delta_2(G)- \frac{\varepsilon}{2} n'\, .\]

We can now use a standard Chernoff-type bound for the hypergeometric distribution
(see for example Lemma 2 in~\cite{HaggkvistThomason95}) to show that the probability that $P_i$ is a
low codegree pair in $G'$ is small.
\begin{align*}
\mathbb{P}\left(d_{G'}(P_i)\leq \frac{n'}{n}\delta_2(G)-\varepsilon
n'\right)&\leq\mathbb{P}\left(X\leq \mathbb{E}(X)-\frac{\varepsilon
n'}{2}\right)\\
&\leq\exp\left(\frac{-{\left(\varepsilon n'/2\right)}^2}{\mathbb{E}(X)/2}\right)\\
&\leq\exp\left(\frac{-\varepsilon^2 n'}{2}\right)\,.
\end{align*}
Summing over all $\binom{n'}{2}$ pairs $P_i$ from $V(G')$ and using the union bound, we deduce
that
\begin{align*}
\mathbb{P}\left({\delta_2(G')\leq \frac{n'}{n}\delta_2(G)-\epsilon
n'}\right)&\leq\binom{n'}{2}\exp\left(\frac{-\varepsilon^2
n'}{2}\right).
\end{align*}
For $n'$ sufficiently large, this is strictly less than $1$. Thus with
strictly positive probability $G'$ satisfies $\delta_2(G')/n'>\delta_2(G)/n-
\varepsilon$ as required -- and in particular a good choice of $G'$ exists.
\end{proof}

With Lemma~\ref{claim:takesmallerconditional} in hand, we can now prove the main
result of this section.
\begin{proposition}\label{prop:limconditionalexists}
For all $3$-graphs $H$ and all families of nonempty $3$-graphs $\mathcal{F}$
not containing $H$, the sequence $\mathrm{coex}(n, \mathcal{F}\vert H)/n$ tends
to a limit as $n \rightarrow \infty$.
\end{proposition}
\begin{proof}
Let $H$ be a $3$-graph and let $\mathcal{F}$ be a family of nonempty $3$-graphs
which does not contain $H$. Set
\[ a_n = \frac{\mathrm{coex}(n, \mathcal{F}\vert H)}{n}\,.\] 
We shall show $(a_n)_{n \in \mathbb{N}}$ is a Cauchy sequence and hence
convergent in $[0,1]$.

Pick $\varepsilon>0$, and let $N=N(\varepsilon, H)$ be the integer whose
existence is guaranteed by Lemma~\ref{claim:takesmallerconditional}. Let $n,n'
\in \mathbb{N}$ be integers with $n \geq n' \geq N$. Suppose $G$ is an
$n$-vertex $\mathcal{F}$-free $3$-graph containing a copy of $H$ with
$\delta_2(G)= \mathrm{coex}(n, \mathcal{F} \vert H)$. By
Lemma~\ref{claim:takesmallerconditional}, $G$ has an $n'$-vertex subgraph $G'$
which contains a copy of $H$ and satisfies $\delta_2(G')/n' \geq \delta_2(G)/n -
\varepsilon$. Since $G$ is $\mathcal{F}$-free, so is $G'$, and we must thus have
 $$
a_n- a_{n'}  \leq a_n - \frac{\delta_2(G')}{n'}\leq a_n -
\frac{\delta_2(G)}{n}+ \varepsilon
= \varepsilon.
 $$

We claim that there also exists an integer $M= M(\varepsilon, H)\geq N$ such that  for all integers $n\geq M$ we have $a_{M} -a_n \leq \varepsilon$. Indeed, either $M_1=N$ is a good choice of $M$ or there exists an integer $M_2>N$ with $a_{M_2}< a_N - \varepsilon$. Then either $M_2$ is a good choice of $M$ or there exists an integer $M_3>M_2$ with $a_{M_3}<a_{M_2}-\varepsilon$, in which case we iterate the argument. As the sequence $a_{M_1}, a_{M_2}, \ldots$ consists of real numbers from $[0,1]$, is strictly decreasing and has gaps between successive terms of at least $\varepsilon$, it can have length at most $1+\lceil 1/\varepsilon \rceil$. Thus after a bounded number of iterations of our argument, we find a good choice of $M$.

Then for any $n \geq M$, we have $\vert a_n -a_M \vert \leq \varepsilon$. It follows that $(a_n)_{n \in \mathbb{N}}$ is Cauchy as claimed, and so converges to a limit in $[0,1]$.
\end{proof}
We may thus define the conditional codegree density of $\mathcal{F}$ given $H$.
\begin{definition}
Let $\mathcal{F}$ be a family of nonempty $3$-graphs, and let $H$ be a $3$-graph not belonging to $\mathcal{F}$. The \emph{conditional codegree density} $\gamma(\mathcal{F}\vert H)$ of $\mathcal{F}$ given $H$ is the limit
\[\gamma(\mathcal{F} \vert H)= \lim_{n \rightarrow \infty} \frac{\textrm{coex}(n, \mathcal{F} \vert H)}{n}.\]
\end{definition}
The following simple observation encapsulates the usefulness of conditional
codegree densities in bounding codegree densities.
\begin{lemma}\label{lem:notationlink}
Let $\mathcal{F}$ be a family of nonempty $3$-graphs and let $H$ be a $3$-graph
not contained in $\mathcal{F}$. Then
\[\gamma(\mathcal{F})=\max\{\gamma(\mathcal{F} \vert
H),\gamma(\mathcal{F}\cup\{H\})\}\,.\]
\end{lemma}
\begin{proof}
Let $c=\max\{\gamma(\mathcal{F}\vert H),\gamma(\mathcal{F}\cup\{H\})\}$. Clearly
we have that $\gamma(\mathcal{F})\geq\gamma(\mathcal{F}\vert H)$ and
$\gamma(\mathcal{F})\geq\gamma(\mathcal{F}\cup\{H\})$, so
$\gamma(\mathcal{F})\geq c$.

Suppose that $(G_n)_{n\in \mathbb{N}}$ is a sequence of $3$-graphs tending
to infinity with $\liminf_{n\rightarrow\infty}\frac{\delta_2(G_n)}{\vert
V(G_n)\vert}> c$. Let $n$ be sufficiently large. Then, since
$\gamma(\mathcal{F}\cup\{H\})\leq c$, $G_n$ must
contain a member of $\mathcal{F}$ or $H$. As $\gamma(\mathcal{F}\vert
H)\leq c$, if $G_n$ contains $H$ then it must contain a member of $\mathcal{F}$
also. In particular, $G_n$ contains a member of
$\mathcal{F}$. It follows that $\gamma(\mathcal{F})\leq c$, as claimed.
\end{proof}

\subsection{Proof of Theorem~\ref{f32density}}
For an integer $t$, the \emph{blow-up} $F(t)$ of a $3$-graph $F$ is the
$3$-graph formed by replacing each vertex $v$ of $F$ by a set $S_v$ of $t$ new
vertices and placing for each $3$-edge $\{x,y,z\}\in E(F)$ all $t^3$ triples
meeting each of $S_x$, $S_y$ and $S_z$ in one vertex. If $\mathcal{F}$ is a family of $3$-graphs then its \emph{blow-up}
$\mathcal{F}(t)$ is defined to be the family $\{F(t):F\in\mathcal{F}\}$.

Just as the ordinary Tur\'an density, the codegree density $\gamma$ exhibits 
\emph{blow-up invariance}: the codegree density of a finite family
is the same as the codegree density of its blow-up. This fact was reproved by
several researchers, see e.g.~\cite{KeevashZhao07, LoMarkstrom12, MubayiZhao07} 
\begin{lemma}[\cite{KeevashZhao07, LoMarkstrom12,
MubayiZhao07}]\label{lem:blowupcodegree}
Let $\mathcal{F}$ be a finite family of $3$-graphs and $t\in \mathbb{N}$. Then
\[\gamma(\mathcal{F}(t))=\gamma(\mathcal{F}).\]
\end{lemma}\qed

Having stated this lemma, let us now define some $3$-graphs we shall need in our
proof of Theorem~\ref{f32density}. Recall from the introduction that $K_4$ is
the complete $3$-graph on four vertices, and $K_4^-$ is the $3$-graph obtained from
$K_4$ by deleting one of its $3$-edges. Further, let $S_k$ denote the
\emph{star} on $k+1$ vertices, that is, the $3$-graph with vertex set
$\{x,y_1,\ldots,y_k\}$ and $3$-edges $\{xy_iy_j:1\leq i<j\leq k\}$. Note that
$S_3$ is (isomorphic to) $K_4^-$.

Finally, let $S_k'$ denote the $3$-graph on $k+2$ vertices obtained by
duplicating the central vertex $x$ of the star $S_k$. Thus $S_k'$ has vertex set
$\{x_1,x_2,y_1,\ldots,y_k\}$ and $3$-edges $\{x_1y_iy_j:1\leq i<j\leq
k\}\cup\{x_2y_iy_j:1\leq i<j\leq k\}$.

Our strategy in the proof of Theorem~\ref{f32density} is to show that if a
$3$-graph $G$ has codegree $\delta_2(G)>
\left(\frac{1}{3}+\varepsilon\right)\vert V(G) \vert$ and $\vert V(G) \vert$ is
large, then $G$ contains a copy of $F_{3,2}$ or it is forced to contain
copies of larger and larger stars. We make this gradual ascension towards
Theorem~\ref{f32density} in a series of lemmas on conditional codegree density,
each of which relies on applying the key Lemma~\ref{lem:extcodegree} with a
suitable weighting $\underline{\alpha}$. We shall repeatedly look for and find copies of $F_{3,2}$ inside
larger $3$-graphs, and it will be convenient to write ``$ab \vert cde$'' to mean
that $abc, abd, abe$ and $cde$ are all $3$-edges (and thus that $\{abcde\}$
spans a copy of $F_{3,2}$).

\begin{lemma}\label{lem:s3'}
$\gamma(F_{3,2},S_3')\leq\frac{1}{3}$.
\end{lemma}
\begin{proof}
Clearly $\gamma(F_{3,2},S_3')\leq\gamma(S_3')$ and since $S_3'$ is a subgraph of
$K_4^{-}(2)$, it is enough by Lemma~\ref{lem:blowupcodegree} to show that
$\gamma(K_4^{-})\leq 1/3$. And indeed $\mathrm{coex}(n,K_4^{-})\leq n/3$ since
if we take any edge $xyz$ in a $K_4^{-}$-free $3$-graph, the neighbourhoods
$\Gamma(x,y)$, $\Gamma(x,z)$, $\Gamma(y,z)$ must be disjoint. Thus
$\gamma(K_4^-)\leq 1/3$ as claimed.
\end{proof}

\begin{lemma}\label{lem:sk'_a}
Let $k \geq 3$. Then $\gamma(F_{3,2}\vert S_k')\leq k/(3k-1)$.
\end{lemma}
\begin{proof}
Suppose $(G_n)_{n \in \mathbb{N}}$ is a $3$-graph sequence tending to infinity
and containing $S_k'$ with
\[\liminf_{n\rightarrow\infty}\frac{\delta_2(G_n)}{\vert V(G_n)
\vert}>\frac{k}{3k-1}\,.\]
Denote the vertices of $S_k'$ by $V(S_k')=\{x_1,x_2, y_1, \ldots  y_k\}$ as
before, and partition the collection of pairs $V(S_k')^{(2)}$ into the three
sets $P_1=\{x_1x_2\}$, $P_2=\{x_iy_j:1\leq i\leq 2,\, 1\leq j\leq k\}$ and
$P_3=\{y_iy_j:1\leq i<j\leq k\}$.

We shall apply Lemma~\ref{lem:extcodegree} using the following weight vector
$\underline{\alpha}\in \Delta(V(S_k')^{(2)})$:
\[\alpha_{uv}=\begin{cases}
\frac{k-1}{3k-1} &\textrm{if~}uv\in P_1\,,\\
\frac{1}{6k-2} &\textrm{if~}uv \in P_2\,,\\
\frac{2}{(k-1)(3k-1)} &\textrm{if~}uv \in P_3\,.
\end{cases}\]

Lemma \ref{lem:extcodegree} guarantees that there is an extension $H$ of $S_k'$
for which
\[w_{\underline{\alpha}}(L(H;S_k'))=\sum_{uv\in
L(H;S_k')}\alpha_{uv}\geq\liminf_{n\rightarrow\infty}\frac{\delta_2(G_n)}{\vert
V(G_n) \vert}>\frac{k}{3k-1}\,,\]
and an infinite subsequence $(G_{n_k})_{k \in \mathbb{N}}$ such that $(G_{n_k})_{k \in \mathbb{N}}$
contains $H$.

We now show that $H$ must contain $F_{3,2}$ to conclude the proof of the lemma.
This is essentially case-checking.
Write $L$ for the set $L(H;S_k')$, $w$ for $w_{\underline{\alpha}} $ and $z$ for
the vertex added to $S_k'$ to form $H$. 
\medskip

\noindent\textbf{Case 1:} suppose that $L$ contains the single pair $x_1x_2$
from $P_1$. If $L$ contains any pair $y_iy_j$ from $P_3$ then $y_iy_j\vert
x_1x_2z$, so that we have a copy of $F_{3,2}$ as claimed. On the other hand if
$P_3$ contains no edge of $L$, then consider $\vert L \cap P_2 \vert$. If this is at
least three, then at least one of the vertices $x_1,x_2$, without loss of
generality $x_1$, must be incident to at least two edges of $L\cap P_2$. Let two
such edges be $x_1y_i$ and $x_1y_j$. Then $zx_1 \vert x_2y_iy_j$, so that again
we have a copy of $F_{3,2}$ as claimed. Finally note that if $L\cap P_3 =
\emptyset$ and $\vert L \cap P_2 \vert \leq 2$ then
 $$
w(L) \leq \frac{ (k-1)\vert L\cap P_1 \vert}{3k-1}+\frac{\vert L \cap P_2
\vert}{2(3k-1)}
\leq  \frac{k}{3k-1},
 $$
contradicting the fact that $w(L)>k/(3k-1)$. Thus we are done in this case.
\medskip

\noindent\textbf{Case 2:} suppose that $L$ does not contain $x_1x_2$, but
contains at least one edge from $P_2$. Without loss of generality let $x_1y_i$
be one such edge.

If $y_i$ is incident to two edges $y_iy_{j_1}$ and
$y_iy_{j_2}$ of $L \cap P_3$, then $zy_i\vert x_1 y_{j_1}y_{j_2}$ and we have a
copy of $F_{3,2}$ as required. On the other hand if $L\cap P_3$ contains at
least one edge $y_{j_1}y_{j_2}$ not incident to $y_i$, then $x_1y_i \vert
zy_{j_1}y_{j_2}$, again spanning a copy of $F_{3,2}$.

Now if $L$ contains exactly one edge $y_iy_j$ from $P_3$ then all edges in $L
\cap P_2$ are incident with one of $y_i,y_j$. In particular, $\vert L \cap P_2
\vert \leq 4$ and
\begin{align*}
w(L)& = \frac{ \vert L\cap P_2 \vert}{2(3k-1)}+\frac{2\vert L \cap P_3
\vert}{(k-1)(3k-1)}\\
&\leq  \frac{2}{3k-1}+ \frac{2}{(k-1)(3k-1)}\\
&=\frac{k}{(3k-1)} \frac{2}{(k-1)}\leq \frac{k}{3k-1} && \textrm{(since $k\geq 3$),}
\end{align*}
a contradiction. On the other hand if $L$ contained no edge from $P_3$, then
 $$
w(L) = \frac{ \vert L\cap P_2 \vert}{2(3k-1)}
\leq  \frac{k}{3k-1}\,,
 $$
again a contradiction of our assumption that $w(L)>k/(3k-1)$.
\medskip

\noindent\textbf{Case 3:} finally, suppose that $L$ contains no edge from $P_1$
or $P_2$. Then $L\subseteq P_3$, and
 $$
w(L)\leq \frac{2\vert P_3 \vert}{(k-1)(3k-1)}
= \frac{k}{3k-1},
 $$
contradicting our assumption that $w(L)> k/(3k-1)$.

It follows that $H$ must contain a copy of $F_{3,2}$, as claimed.
\end{proof}

\begin{lemma}\label{lem:sk'_b}
Let $k\geq 3$. Then $\gamma(F_{3,2},S_{k+1},K_4\vert S_k')\leq 1/3$.
\end{lemma}
\begin{proof}
This is very similar to the proof of Lemma \ref{lem:sk'_a}. Suppose $(G_n)_{n
\in \mathbb{N}}$ is a $3$-graph sequence tending to infinity which contains
$S_k'$ and satisfies \[\liminf_{n\rightarrow\infty}\frac{\delta_2(G_n)}{\vert
V(G_n) \vert }>\frac{1}{3}\,.\] 
Denote the vertices of $S_k'$ by $V(S_k')=\{x_1,x_2,y_1,\ldots,y_k\}$ as before
and partition ${V(S_k')}^{(2)}$ into the three sets $P_1=\{x_1x_2\}$,
$P_2=\{x_iy_j:1\leq i\leq 2,\, 1\leq j\leq k\}$ and $P_3=\{y_iy_j:1\leq i<j\leq
k\}$.

We apply Lemma~\ref{lem:extcodegree} with a slightly different weighting. Let
$\underline{\alpha}$ be defined by:
\[\alpha_{uv}=\begin{cases}
\frac{k-2}{3(k-1)} &\textrm{if~}uv\in P_1\,,\\
\frac{1}{6(k-1)} &\textrm{if~}uv \in P_2\,,\\
\frac{2}{3k(k-1)} &\textrm{if~}uv \in P_3\,.
\end{cases}\]
Lemma~\ref{lem:extcodegree} guarantees the existence of an extension $H$ of
$S_k'$ with
\[w_{\underline{\alpha}}(L(H;S_k'))=\sum_{uv\in
L(H;S_k')}\alpha_{uv}\geq\liminf_{n\rightarrow\infty}\frac{\delta_2(G_n)}{\vert
V(G_n)\vert}>\frac{1}{3}\,,\]
and of an infinite subsequence $(G_{n_k})_{k \in \mathbb{N}}$ such that
$(G_{n_k})_{k \in \mathbb{N}}$ contains $H$.

We now show that any such extension $H$ must contain either $F_{3,2}$, $S_{k+1}$
or $K_4$. As in the previous lemma, this is just a matter of case-checking.
Write $L$ as before for the set $L(H;S_k')$, $w$ for $w_{\underline{\alpha}} $
and $z$ for the vertex added to $S_k'$ to form $H$. 
\medskip

\noindent\textbf{Case 1:} suppose $x_1 x_2 \in L$. By the analysis in Case~1 of
Lemma~\ref{lem:sk'_a}, we know that if $L$ contains any edge from $P_3$ or at
least three edges from $P_2$ then $H$ contains a copy of $F_{3,2}$ and we are
done. On the other hand if neither of these happen then
 $$
w(L) = \frac{(k-2)\vert L\cap P_1 \vert}{3(k-1)}+ \frac{\vert L\cap P_2
\vert}{6(k-1)} \leq \frac{k-2}{3(k-1)}+ \frac{1}{3(k-1)}= \frac{1}{3},
 $$
contradicting our assumption that $w(L)>1/3$.
\medskip

\noindent\textbf{Case 2:} suppose $x_1x_2 \notin L$, but $L\cap P_2 \neq
\emptyset$. By the analysis in Case~2 of Lemma~\ref{lem:sk'_a}, we know that if
$L$ contain an edge from $P_2$ incident to two edges from $P_3$ or an edge from
$P_2$ and a disjoint edge from $P_3$, then $H$ contains a copy of $F_{3,2}$ and
we are done.

Also if $L$ contains an edge $y_{j_1}y_{j_2}$ of $P_3$ and two edges
$x_iy_{j_1}$, $x_iy_{j_2}$ from $P_2$ then $zx_i y_{j_1}y_{j_2}$ forms a copy of
$K_4$, and we are done. In addition if for some $i\in\{1,2\}$ $L$ contains all $k$ edges of the form
$x_i y_j$ then $x_i, z, y_1, \ldots y_k$ forms a copy of $S_{k+1}$, and we are
done.

Now let us suppose none of these things happens. If $L$ contains an edge from
$P_3$ then $\vert L \cap P_2 \vert \leq 2$ and $\vert L \cap P_3 \vert\leq 1 $
(else we have a copy of $K_4$ or $F_{3,2}$) and thus
\begin{align*}
w(L)&\leq \frac{2}{6(k-1)}+ \frac{2}{3k(k-1)}\\
&<1/3 && \textrm{(since $k\geq 3$),}
\end{align*}
a contradiction. On the other hand if $L$ contains no edge from $P_3$ then
$\vert L \cap P_2 \vert \leq 2(k-1)$ (else we have a copy of $S_{k+1}$) and
\begin{align*}
w(L)&\leq \frac{2(k-1)}{6(k-1)}=1/3\,,
\end{align*}
again a contradiction.
\medskip

\noindent\textbf{Case 3:} finally suppose $L$ contains no edge from $P_1$ or
$P_2$. Then $L \subseteq P_3$ and 
$$
w(L)\leq \frac{2\binom{k}{2}}{3k(k-1)}=1/3\,,
$$
contradicting yet again our assumption that $w(H)>1/3$.

It follows that $H$ must contain a copy of one of $F_{3,2}$, $K_4$ or $S_{k+1}$,
as claimed.
\end{proof}

\begin{lemma}\label{lem:k4}
$\gamma(F_{3,2}\vert K_4(2))\leq 1/3$.
\end{lemma}
\begin{proof}
We shall in fact prove the slightly stronger statement that $\gamma(F_{3,2}\vert
K_4'')\leq 1/3$, where $K_4''$ is the $3$-graph on $6$ vertices
$\{a,b,c_1,c_2,d_1,d_2\}$ with edges $\{abc_i:\ i\in[2]\}\cup\{abd_i:\ i\in
[2]\}\cup\{ac_id_j:\ i,j\in [2]\}\cup\{bc_id_j:\ i,j\in[2]\}$. In other words,
$K_4''$ is the $3$-graph formed by duplicating two distinct vertices of $K_4$
(and hence a subgraph of $K_4(2)$).

Suppose that $(G_n)_{n \in \mathbb{N}}$ is a $3$-graph sequence tending to
infinity which contains $K_4''$ and satisfies
\[\liminf_{n\rightarrow\infty}\frac{\delta_2(G_n)}{\vert
V(G_n)\vert}>\frac{1}{3}\,.\] 
We apply Lemma~\ref{lem:extcodegree} once more, with the following weighting
$\underline{\alpha}$:
\[\alpha_{uv}=\begin{cases}
\frac{1}{6} &\text{if~}uv\in \{ac_1,ad_1,bc_1,bd_1,c_1c_2,d_1d_2\}\,,\\
0 &\text{otherwise}\,.
\end{cases}\]
Lemma~\ref{lem:extcodegree} guarantees the existence of an extension $H$ of
$K_4''$ with
\[w_{\underline{\alpha}}(L(H;K_4''))=\sum_{uv\in
L(H;K_4'')}\alpha_{uv}\geq\liminf_{n\rightarrow\infty}\frac{\delta_2(G_n)}{\vert
V(G_n)\vert}>\frac{1}{3}\,,\]
and of an infinite subsequence $(G_{n_k})_{k \in \mathbb{N}}$ such that
$(G_{n_k})_{k \in \mathbb{N}}$ contains $H$.

We now show that any such extension $H$ contains a copy of $F_{3,2}$ as a
subgraph. Write again $L$ for the set $L(H;K_4'')$, $w$ for
$w_{\underline{\alpha}} $ and $z$ for the vertex added to $K_4''$ to form $H$.

Since $w(L)>1/3$, at least three of the edges in
$\{ac_1,ad_1,bc_1,bd_1,c_1c_2,d_1d_2\}$ must be contained in the link graph $L$.
If the three edges in that set which are incident to $c_1$ are in $L$, then
$zc_1\vert c_2ab$ and we have a copy of $F_{3,2}$. Also if $c_1c_2 \in L$ and
$L$ contains either $ad_1$ or $bd_1$ then we have either $ad_1\vert c_1c_2z$
or $bd_1 \vert c_1c_2z$, and thus we have a copy of $F_{3,2}$. Similarly if
$d_1d_2 \in L$ and either $ac_1$ or $bc_1$ are in $L$ then we have $ac_1
\vert d_1d_2z$ or $bc_1 \vert d_1d_2z$.

It follows in particular that if $L$ contains $c_1c_2$ then we have a copy of
$F_{3,2}$. In exactly the same way we are done if $d_1d_2 \in L$. So finally
suppose that neither of $c_1c_2$ and $d_1d_2$ is contained in $L$. Then at least
three of the four edges $ac_1$, $ad_1$, $bc_1$, $bd_1$ must be in. In particular
we must contain a pair of non-incident edges from that set. Assume without loss
of generality that $ad_1$ and $bc_1$ are both in. Then $ad_1 \vert bc_1z$, so
that we have again a copy of $F_{3,2}$, as claimed.
\end{proof}

With Lemmas~\ref{lem:s3'}, \ref{lem:sk'_a}, \ref{lem:sk'_b} and \ref{lem:k4} in
hand, we can finally prove our codegree density result.
\begin{proof}[Proof of Theorem~\ref{f32density}]
We first show by induction on $k$ that $\gamma(F_{3,2}, S_k')\leq 1/3$ for all
$k\geq 3$.

For the base case, we know from Lemma~\ref{lem:s3'} that $\gamma(F_{3,2},
S_3')\leq 1/3$. For the inductive step, suppose we knew that $\gamma(F_{3,2},
S_K')\leq 1/3$ for some $K\geq 3$. We know from Lemma~\ref{lem:sk'_b} that
$\gamma(F_{3,2}, K_4, S_{K+1}\vert S_K')\leq 1/3$. It then follows by
Lemma~\ref{lem:notationlink} that 

\begin{align*} 
\gamma(F_{3,2}, K_4, S_{K+1})&=\max\Bigl(\gamma\left(F_{3,2}, K_4, S_{K+1}, S_K'\right), \gamma\left(F_{3,2}, K_4, S_{K+1}\vert S_K'\right) \Bigr)\\
&\leq \max\left( \gamma\left(F_{3,2}, S_K'\right), \frac{1}{3}\right)\leq \frac{1}{3}.
\end{align*}

Using blow-up invariance (Lemma~\ref{lem:blowupcodegree}), we deduce that
$\gamma(F_{3,2}, K_4(2), S_{K+1}')\leq 1/3$. Combining this with the result of
Lemma~\ref{lem:k4} that $\gamma(F_{3,2}\vert K_4(2))\leq 1/3$, we have by one
more application of Lemma~\ref{lem:notationlink} that $\gamma(F_{3,2},
S_{K+1}')\leq 1/3$.

It follows that $\gamma(F_{3,2}, S_k')\leq 1/3$ for all $k \geq 3$, as claimed.
Our codegree density result is straightforward from this: for any $k\geq 3$ we
have by Lemma~\ref{lem:notationlink} that
\[\gamma(F_{3,2})= \max\left(\gamma(F_{3,2}\vert S_k'), \gamma(F_{3,2},
S_k')\right)\,.\]
We also know from Lemma~\ref{lem:sk'_a} that $\gamma(F_{3,2} \vert S_k')\leq
k/(3k-1)$. Since as shown inductively above we have $\gamma(F_{3,2},S_k')\leq
1/3$ for all $k\geq 3$, it follows that 
\[\gamma(F_{3,2})\leq \inf_{k \geq 3} \left(\max\left(\frac{k}{3k-1},
\frac{1}{3}\right)\right)=\frac{1}{3}\,,\]
as desired.
\end{proof}

\section{Codegree density and stability via flag algebras}\label{flagsection}

In this section, we use the flag algebra method of
Razborov~\cite{Razborov07,Razborov10} to give a second proof of
Theorem~\ref{f32density}
and to obtain the stability result claimed in Theorem~\ref{f32stability}.
Several good expositions of flag algebras from an extremal combinatorics
perspective have already appeared in the literature~\cite{BaberTalbot11,hirst:4vertex,
FalgasRavryVaughan13, Keevash11}. We shall therefore be rather brief,
directing the reader to the aforementioned papers
for details. Our proof is generated by computer using Vaughan's \emph{Flagmatic} package (version 2.0) ~\cite{flagmatic20}.
A proof certificate is stored
under the name \texttt{F32Codegree.js} in the ancillary folder
of the arxiv version of this paper~\cite{falgasravry+marchant+pikhurko+vaughan:arxiv}, which also contains
the flagmatic code \texttt{F32Codegree.sage} that generated the certificate. In Section~\ref{cert}
we describe the structure of the file
\texttt{F32Codegree.js} and show how the information contained therein implies the desired bound
$\gamma(F_{3,2})\le \frac13$. Since the file is large (over 2MB) and contains integers with dozens of digits, verification of the proof
requires a computer as well. In order to verify all stated properties of the proof certificate, the reader can write her own script, or use the script \texttt{inspect\_certificate.py} included in \emph{Flagmatic} to do some of the verifications for her.

\subsection{Structure of the proof certificate}\label{cert}

First of all, we refer the reader to the \emph{Flagmatic} User's
Guide~\cite{Vaughan13} that, among
many other things, describes how combinatorial structures (including types and
flags that are defined below) are stored in proof certificates.

The certificate consists of various parts. Here we describe only those that are
directly needed for verifying the validity of our proof.

\jSection{admissible\_graphs} lists all $F_{3,2}$-free $3$-graphs on $N=6$ 
vertices up to isomorphism. There are exactly 426 of them; let us denote them
by $G_1,\dots,G_{426}$. 

\jSection{types} lists \emph{types} with $2\ell<N$
vertices, i.e.\ (vertex-labelled) $F_{3,2}$-free
$3$-graphs with vertex set $\emptyset$, $[2]$ and $[4]$. For our application,
we need only one representative from each class of isomorphic
3-graphs; thus the number of listed types of order 0, 2 and 4 is
respectively $1$, $1$, and $5$. Let us denote them by $\tau_1,\dots,\tau_7$,
using the same ordering as in \emph{Flagmatic}: first by the number of vertices
and then lexicographically by the list
of $3$-edges. For example, $\tau_2$ is the type with 2 (labelled) vertices
and no 3-edges while $\tau_7$ is a vertex-labelled $K_4^3$.

For a type $\tau$ on $[k]$, a
\emph{$\tau$-flag} is a $(k+1)$-tuple $(F,x_1,\dots,x_k)$ where $F$ is an
$F_{3,2}$-free $3$-graph and $x_1,\dots,x_k\in V(F)$ are distinct vertices
of $F$ such that the map $i\mapsto x_i$ is an isomorphism
between $\tau$ and the induced subgraph $F[\{x_1,\dots,x_k\}]$. We can view
a flag as a 3-graph with $k$ labelled roots that
induce a copy of $\tau$ (while the remaining vertices are treated as
unlabelled). This leads to the natural definition of an \emph{isomorphism} $f$
between two $\tau$-flags $(F,x_1,\dots,x_k)$ and $(H,y_1,\dots,y_k)$: namely
an isomorphism $f$ between the unlabelled $3$-graphs $F$
and $H$ such that the roots are preserved, that is, 
$f(x_i)=y_i$ for every $i\in [k]$.

\jSection{flags}
contains for each $t\in [7]$ the list of all
$\tau_t$-flags $F_1^{\tau_t},\dots,F_{g_t}^{\tau_t}$ with $(N+\vert V(\tau_t)\vert)/2$
vertices up to
flag isomorphism. For
example, if $t=1$, then $\tau_t$ is the type with no vertices, and we have to
list all unlabelled 3-graphs of order 3; clearly, there are exactly two of them
(edge and non-edge).
If $t=2$, then $\tau_t$ is the (unique) 2-vertex type, and we have to list all
4-vertex 3-graphs $G$ with two roots; for $e(G)=0,1,2,3,4$ there
are respectively $1,3,4,3,1$ non-isomorphic ways of placing the roots. Thus
$g_2=12$. 

For each $i\in [7]$, the certificate (indirectly) contains a symmetric
$\left(g_i\times g_i\right)$-matrix $Q^{\tau_i}$. More precisely, $Q^{\tau_i}=RQ'R^T$ where
$Q'$ is a diagonal matrix all of whose diagonal entries are positive rational
numbers (listed in \jsection{qdash\_matrices}) and $R$ is a rational matrix
(listed in \jsection{r\_matrices}). This
representation automatically implies that the matrix $Q^{\tau_i}$ is
positive
semi-definite.

\jSection{axiom\_flags} lists all $\tau_2$-flags with $5$ vertices. Recall that
$\tau_2$ is the (unique) type with $2$ labelled vertices. There are $154$ such flags. Let us
denote them by $M_1,\dots,M_{154}$. \jSection{density\_coefficients} lists
non-negative rational numbers $c_1,\dots,c_{154}$, one for each flag $M_i$.

Let $\tau$ be a type on $[k]$. For two $\tau$-flags $(F,x_1,\dots,x_k)$ 
and $(H,x_1,\dots,x_k)$ let 
 $$
 \pp((F,x_1,\dots,x_k),(H,y_1,\dots,y_k))
 $$
be the number of $\vert V(F) \vert $-sets $X$ such that 
$\{y_1,\dots,y_k\}\subseteq X\subseteq V(H)$
and the induced $\tau$-flag $(H[X],y_1,\dots,y_k)$
is isomorphic to the $\tau$-flag $(F,x_1,\dots,x_k)$. For example,
$\pp((K_3^3,x_1,x_2),(G,y,z))$ is the codegree of $(y,z)$ in $G$, where
$(K_3^3,x_1,x_2)$ is the single $3$-edge with two roots.

Let $G$ be an arbitrary $F_{3,2}$-free $3$-graph of (large) order $n$. 

First, we compute two
parameters $\sigma_1$ and $\sigma_2$ of $G$ using the information above.
We let 
 \begin{equation}\label{eq:b}
 \sigma_1=\sum_{x_1,x_2}
\left(\pp\big(
(K_3^3,x_1,x_2),(G,x_1,x_2)\big)-\frac n3\right)\sum_{i=1}^{154}
c_i \pp\big(M_i,(G,x_1,x_2)\big),
 \end{equation}
 where the sum is over all $n(n-1)$ choices of distinct ordered pairs $(x_1,x_2)$ from $V(G)$. Note that if the minimum codegree of $G$ is at least $n/3$ then $\sigma_1 \geq 0$.

The definition of $\sigma_2$ is slightly more complicated.
Initially, set $\sigma_2=0$. 
Then for each $k \in\{0,2,4\}$ let us do the following.
Enumerate all
$n(n-1)\dots (n-k+1)$ sequences $(x_1,\dots,x_k)$ of distinct
vertices in $V(G)$. If the induced type 
$(G[\{x_1,\dots,x_k\}],x_1,\dots,x_k)$ is isomorphic to some $\tau_i$, then
we add
$\B p Q^{\tau_i}\B p^T$ to
$\sigma_2$, where
 \begin{equation}\label{eq:x}
 \B p=\big(\pp(F_1^{\tau_i},(G,x_1,\dots,x_k)),
\dots,\pp(F_{g_i}^{\tau_i},(G,x_1,\dots,x_k))\big).
 \end{equation}
 Since each $Q^{\tau_i}$ is positive semi-definite, we have that $\B p
Q^{\tau_i}\B p^T\ge 0$. Thus $\sigma_2$ is non-negative.

Let us take some type $\tau$ on $[k]$ and two $\tau$-flags $F_1$ and $F_2$
with respectively $\ell_1$ and $\ell_2$ vertices. Let $\ell=\ell_1+\ell_2-k$.
Consider the sum
 \begin{equation}\label{eq:prod}
 \sum_{x_1,\dots,x_k} \pp(F_1,(G,x_1,\dots,x_k))\, \pp(F_2,(G,x_1,\dots,x_k))
 \end{equation} 
 over
all choices of $k$-tuples $(x_1, \ldots x_k)$ that induce a copy of $\tau$ in $G$. Each term $\pp(F_i,(G,x_1,\dots,x_k))$ in (\ref{eq:prod}) can be
expanded as the sum over $\ell_i$-sets $X_i$ with
$\{x_1,\dots,x_k\}\subseteq X_i \subseteq V(G)$ of the indicator
function that $X_i$ induces a $\tau$-flag isomorphic to $F_i$. Ignoring the
choices
when $X_1$ and $X_2$ intersect outside of $\{x_1,\dots,x_k\}$, the remaining
terms can be generated by choosing an $\ell$-set $X=X_1\cup X_2$ first, then
distinct
$x_1,\dots,x_k\in X$ to form $X_1 \cap X_2$, and finally splitting the remaining vertices of $X$ between $X_1$ and $X_2$ so that $\vert X_i\vert=\ell_i$. Clearly, the terms that
we ignore contribute at most $O(n^{\ell-1})$ in total. Also, the contribution
of each $\ell$-set $X$ depends only on the isomorphism class of $G[X]$.
Thus the sum
in~(\ref{eq:prod})
can be written as an explicit
linear combination of the subgraph counts $\pp(H,G)$, where $H$ runs
over unlabelled $3$-graphs with $\ell$ vertices, modulo
an additive error term $O(n^{\ell-1})$. An explicit formula for computing
this linear combination can be found in e.g.~\cite[Lemma~2.3]{Razborov07}. 

Thus if we expand
each quadratic form $\B p Q^{\tau_i}\B p^T$ and take the 
sum over all suitable $x_1,\dots,x_k\in V(G)$, where $k=\vert V(\tau_i)\vert$, then we 
obtain a (fixed) linear combination of $\pp(G_1,G),\dots,\pp(G_{426},G)$
with an additive error term of $O(n^5)$. The analogous claim holds for each term
in the right-hand side of (\ref{eq:b}). Thus both $\sigma_1$ and $\sigma_2$ can be
represented in this form, that is, 
 \begin{equation}\label{eq:a+b}
 \sigma_1+\sigma_2 = \sum_{i=1}^{426} \alpha_i \pp(G_i,G) + O(n^5),
 \end{equation}
 where each $\alpha_i$ is a rational number that does not depend on $n$
and that can be computed given the 
information above (namely the matrices $Q^{\tau_j}$ and the coefficients $c_j$). An explicit formula
for
$\alpha_i$ is rather messy, so we do not state it.

The crucial properties that our certificate possesses is that each $\alpha_i$ is non-positive and that $c_2>0$ for the $\tau_2$-flag \flag{5:123(2)} (listed as $M_2$ in \jSection{axiom\_flags}), which in \emph{Flagmatic} notation denotes the 5-vertex 3-graph with one 3-edge and two vertices of that 3-edge labelled. These properties (involving rational numbers) can be verified by the scripts that come with \emph{Flagmatic} and use exact arithmetic. Explicitly, the $\alpha_i$ are stored in an array by \emph{Flagmatic}, called \texttt{problem.\_bounds}. Asking sage to list all strictly positive elements in that array returns the empty set. As for the value of $c_2$, this can be read out by using the \texttt{varproblem} script. We refer the reader to the file \texttt{F32Codegree.sage} that contains such a verification at the end.

Assuming the above properties, we are ready to prove that $\gamma(F_{3,2})\le
\frac13$. Suppose on the contrary that $\gamma(F_{3,2})>1/3+c$ for some $c>0$. 

Let $\varepsilon$ be an arbitrary real with $0<\varepsilon<\frac{1}{20}$, and let $n$ be sufficiently large.
Pick an $F_{3,2}$-free $3$-graph $G$ of order $n$
and minimum codegree at least $(\frac13+c)n$. Given $G$, compute
$\sigma_1$ and $\sigma_2$ as above. We already know that $\sigma_2\ge 0$.
Also, as remarked earlier, the codegree assumption implies
that each summand in (\ref{eq:b}) is non-negative, so that $\sigma_1\ge 0$. 

\begin{lemma}\label{lm:cj>0} Let $j\in[154]$ be such that $c_j>0$. 
Write $M_j^0$ for the unlabelled version of $M_j$. Then $\pp(M_j^0,G)< \varepsilon
{n\choose 5}$.\end{lemma}

\begin{proof} 

Let us derive a contradiction from assuming that $\pp(M_j^0,G)\ge \varepsilon
{n\choose 5}$.
For each $5$-set $X\subseteq V(G)$
that induces $M_j^0$, choose $x_1,x_2\in X$ such that the induced
$\tau_2$-flag $(G[X],x_1,x_2)$ is
isomorphic to $M_j$. The number of pairs $(x_1,x_2)$ that appear
for at least $\varepsilon^2 {n-2\choose 3}$ different choices of $X$ is at least
$\varepsilon^2{n\choose 2}$: indeed, otherwise the number of sets $X$ as above is at most 
 $$
 \varepsilon^2{n\choose 2} \times {n\choose 3}+ {n\choose 2}\times \varepsilon^2 {n-2\choose 3}
 <\varepsilon{n\choose 5}
 $$
for $n$ sufficiently large (since $\varepsilon < \frac{1}{20}$), a contradiction. Each of these $\varepsilon^2{n\choose 2}$ pairs $(x_1,x_2)$ contributes
at least $cn\times c_j\varepsilon^2 {n-2\choose 3}$ to (\ref{eq:b}). 
Thus $\sigma_1=\Omega(n^6)$, which contradicts (\ref{eq:a+b}). 
(Recall that $\sigma_2\ge 0$ while each $\alpha_j\le 0$.)
\end{proof}

Since $\varepsilon>0$ was arbitrary it follows that our hypothetical
counterexample $G$ satisfies $\pp(M_j^0,G)=o(n^5)$ for each $j\in [154]$
with $c_j>0$. In particular, $\pp(H,G)=o(n^5)$, where $H$ is the 5-vertex
3-graph with exactly one edge.

We now use the \emph{random sparsification} trick, as in \cite[Section 4.3]{glebov+kral+volec}. Namely, fix $p$ with $0<p<\min \left(\frac{c}{4}, \frac{1}{2}\right)$ and let $G'$ be obtained from $G$ by deleting each edge with probability $p$. Then it is not hard to show (cf Lemma~\ref{claim:takesmallerconditional}) that with high probability, $\delta_2(G')\ge (1/3+c-2p)n>(1/3+c/2)n$. We know that $G'$ is $F_{3,2}$-free (since $G$ is). Also, as $\vert E(G) \vert =\Omega(n^3)$, 
$G$ has $\Omega(n^5)$ $5$-sets that span at least one edge. Each such set produces a copy of $H$ in $G'$ with probability at least $p^{{5\choose 3}}$, which is small but strictly positive. In particular, with high probability $\pp(H,G')=\Omega(n^5)$: a typical outcome $G'$ leads to a contradiction. Thus $\gamma(F_{3,2})\le \frac13$ as claimed. \qed

\subsection{Generating the certificate}

Although we have formally verified that $\gamma(F_{3,2})\le \frac13$, let us
briefly describe the steps that led to the certificate. As we already mentioned, the ancillary folder
of~\cite{falgasravry+marchant+pikhurko+vaughan:arxiv} also contains the
flagmatic code \texttt{F32Codegree.sage} that generated it as well as the
transcript of the whole session (file \texttt{F32Codegree.txt}).

The method of using positive semi-definite matrices
$Q^{\tau_i}$ to obtain inequalities between subgraph densities is fairly
standard by now and has been used for a number of other problems. The new
ingredient is the (rather obvious) idea to use~(\ref{eq:b}) for
deriving consequences of the codegree assumption $\delta_2(G)\ge \frac13 n$, namely that $\sigma_1\ge 0$ for any choice of non-negative coefficients $c_i$. The verification that 
each $\alpha_i$ can be made non-positive can be done
via semi-definite programming. More specifically, one can create an unknown
block-diagonal matrix $X\succeq 0$ whose blocks are
$Q^{\tau_1},\dots,Q^{\tau_7}$, followed by $c_1,\dots,c_{154}$ as diagonal
entries. Also, we added the extra restriction
$c_1+\dots+c_{154}=1$, to avoid the trivial solution when all unknowns are
zero. This is done
automatically by the function \texttt{make\_codegree\_problem}. The full
support of general `axioms' (such as the codegree
assumption) is not implemented in Version 2.0 of
\emph{Flagmatic}. Hopefully, this will be done in future releases.

The choice $N=6$ came from experimenting with the above approach (as $N=5$
was not enough). Our experiments also suggested that the types $\tau_1$ (empty
vertex set) and $\tau_5$ (two 3-edges on 4 vertices) are not
really needed, that is, we can let $Q^{\tau_1}$ and $Q^{\tau_5}$ be the zero
matrices (thus making the rounding step easier as we will have fewer
parameters). This was done by the command \texttt{set\_inactive\_types}.

A crucial observation for the rounding procedure is that any flag algebra proof as above has to satisfy some relations. Namely, if we run our flag algebra argument
on an almost extremal example $G=T_{V_1,V_2,V_3}$ with $\vert V_i\vert=n/3$, then all the inequalities we obtain are
tight up to an $O(n^5)$ additive error. This has a number of consequences.

Call a $3$-graph $G_i$ of order $6$ \emph{sharp} if $\alpha_i=0$. The following lemma tells us a number of graphs must necessarily be sharp.
\begin{lemma}\label{lm:Sharp}
If a $6$-vertex $3$-graph $G_i$ is isomorphic to an induced subgraph of some $T_{A,B,C}$ construction,
then $G_i$ is sharp.
\end{lemma}
\begin{proof} 
Let $G$ be a balanced $T_{A,B,C}$ construction on $n$ vertices. Since $G_i$ is an induced $6$-vertex subgraph of a $T_{A,B,C}$ construction, it readily follows that $P(G_i, G)=\Omega(n^6)$. Now the minimum codegree in $G$ is at least $n/3-2$, whence $\sigma_1(G)\geq -O(n^5)$. By definition, $\sigma_2(G)\geq 0$. Thus we have $\sigma_1(G)+\sigma_2(G)\geq -O(n^5)$. Since $\alpha_j \leq 0$ for all $j \in [426]$, equality (\ref{eq:a+b}) then implies that $-O(n^5) \leq \alpha_i P(G_i, G)$. As $P(G_i, G)=\Omega(n^6)$, we must have $\alpha_i =0$, as claimed.
\end{proof} 

\begin{lemma}\label{lm:Eigenvectors} Let $\tau_i$ be a type on $k\in \{0,2,4\}$ vertices $x_1,\dots,x_k$ which appears as an induced subgraph in a $T_{A,B,C}$ construction. 

Form $\B p$ as in (\ref{eq:x}), with $G$ a balanced $T_{A,B,C}$ construction on $n$ vertices, and write $\| \B p \|$ for its $\ell_2$ norm. Then the limit of $\B p/\|\B p\|$ as $n\to\infty$ is
a zero eigenvector of $Q^{\tau_i}$. 
\end{lemma}
 \begin{proof} 
 Let $G$ be a balanced $T_{V_1,V_2,V_3}$ construction on $n$ vertices. The codegrees of pairs from  $V(G)$ vary between $\lfloor n/3\rfloor -1$ and $\lceil n/3 \rceil$, so that $\vert \sigma_1(G) \vert =O(n^5)$. Now, for all $G_i$ which are $6$-vertex subgraphs of $G$ we have by Lemma~\ref{lm:Sharp} above that $\alpha_i=0$, while for all other $6$-vertex $3$-graphs $G_i$ we have $P(G_i, G)=0$. Equality (\ref{eq:a+b}) thus tells us that $O(n^5)+\sigma_2(G)=O(n^5)$, whence we deduce that $\sigma_2(G)=O(n^5)$.

Now, for each $k \in \{0,2,4\}$ there are $3^k$ sequences $\boldsymbol{\epsilon} =(\epsilon_1, \epsilon_2, \ldots \epsilon_k)$ with $\epsilon_i \in \{1,2,3\}$. Call a sequence of vertices $(x_1, \ldots x_k)$ an \emph{$\boldsymbol{\epsilon}$-sequence} if $x_i \in V_{\epsilon_i}$ for every $i$. For every $\boldsymbol{\epsilon} \in \{1,2,3\}^k$ there exists a unique type $\tau_i$ (which, obviously, embeds into $T_{A,B,C}$ constructions) such that for every $\boldsymbol{\epsilon}$-sequence $(x_1, \ldots x_k)$, $(G[\{x_1, \ldots x_k\}], x_1, \ldots x_k)$ is isomorphic to $\tau_i$. What is more, for every such $\boldsymbol{\epsilon}$-sequence the vector $\B p$ formed as in (\ref{eq:x}) is identical (depends on $\boldsymbol{\epsilon}$ but not on the choice of the $x_i$).

Fix $\boldsymbol{\epsilon} \in \{1,2,3\}^k$. By the non-negativity of the summands contributing to $\sigma_2(G)$, we deduce that the sum of $ \B p Q^{\tau_i} \B p^T$ over all $\boldsymbol{\epsilon}$-sequences is at most $O(n^5)$. Now this latter sum consists of $\Omega(n^k)$ identical terms, and $\| \B p\|=\Omega(n^{3-\frac{k}{2}})$. It follows that 
\begin{align*}
0 \leq \frac{\B p}{\|\B p\|}Q^{\tau_i} \frac{\B p^T}{\|\B p\|} &=\B p Q^{\tau_i} \B p^T \times O(n^{k-6})\\
& \leq O\left(\frac{\sigma_2(G)}{n^k}\right)\times O(n^{k-6})\\
&= O(n^{-1})=o(1).
\end{align*}
It is straightforward to see that for each $\boldsymbol{\epsilon} \in \{1,2,3\}^k$, the (unique) vector $\B p / \| \B p \|$ which can be formed from $\boldsymbol{\epsilon}$-sequences converges to a limit as $n\rightarrow \infty$. It follows from the inequality above and the positive semi-definiteness of $Q^{\tau_i}$ that this limit is a zero eigenvector of $Q^{\tau_i}$, as claimed. \end{proof}

In addition to the above, some further `forced' identities can be derived.
\begin{lemma}\label{lm:PhamtomSharp} Let $T'$ be obtained from  a $T_{V_1,V_2,V_3}$ construction with $\vert V_i \vert\ge 6$ for each $i$ by adding an extra `tripartite' $3$-edge $\{u_1,u_2,u_3\}$ with $u_i \in V_i$.
If a $6$-vertex $3$-graph $G_i$  is isomorphic to an induced subgraph of $T'$,
then 
$G_i$ is sharp.
 \end{lemma}

\begin{proof} 
We may assume that $G_i$ contains the tripartite $3$-edge $\{u_1,u_2,u_3\}$, for otherwise it is isomorphic to an induced subgraph of $T_{V_1,V_2, V_3}$ and we are done by Lemma~\ref{lm:Sharp}.

Now, let
$G$ be obtained from $T_{V_1,V_2,V_3}$ with $\vert V_1\vert=\vert V_2 \vert =\vert V_3 \vert =n/3$ by adding
the complete $3$-partite $3$-graph with parts $U_1\cup U_2\cup U_3$, where 
$U_i\subseteq V_i$ has size $\e n$ for some small $\e>0$. This $3$-graph is
not $F_{3,2}$-free but nothing prevents us from computing $\sigma_1$ and
$\sigma_2$ (which are still nonnegative) using the same formulae as before. When we expand
$\sigma_1+\sigma_2$ as in (\ref{eq:a+b}), the coefficients
$\alpha_1,\dots,\alpha_{426}$ will be the same but we will have an extra sum
$\sum_H \beta_H\pp(H,G)$ where $H$ runs over $6$-vertex $3$-graphs, each
containing a copy of $F_{3,2}$. While we have no control over the sign of
each $\beta_H$, we know that they are constants independent of $n$. Also, we have $\pp(H,G)\le (3\e)^4 n^6$. (Indeed, 
each $H$-subgraph of $G$ has to use at least 4 vertices
from $U=U_1\cup U_2\cup U_3$ because each copy of $F_{3,2}\subseteq G$ uses
at least two added edges.)

Since $\e$ can be arbitrarily small, 
the terms of order $O(\e^3n^6)$ in the new version of (\ref{eq:a+b}) should
have correct signs to avoid a contradiction. (There are no new terms of order $\e
n^6$ or $\e^2 n^6$, as we need to hit at least three
vertices of $U$ to detect an added 3-edge.) For our $G_i$, we have that
$\pp(G_i,G)=\Omega(\e^3n^6)$. Indeed, take an arbitrary embedding $f: V(G_i)\to V(G)$ and modify it to obtain an embedding $f'$ such that for every $x\in V(G_i)$, $f'(x),f(x)$ are always in the same part $V_i$ and $f'(x)\in U_i$ if and only if $f(x)\in U_i$. The resulting map $f': V(G_i)\to V(G)$ gives us another embedding of $G_i$ into $G$. Clearly, there are at least $(1-o(1))(\e n)^3 (n/3)^3$ possible ways to choose $f'$. Thus necessarily $\alpha_i=0$ (otherwise we would violate the non-negativity of $\sigma_1+\sigma_2$), and $G_i$ is sharp as claimed.\end{proof}

We call the additional $3$-edge $\{u_1,u_2,u_3\}$ in Lemma~\ref{lm:PhamtomSharp} a \emph{phantom edge}. Such edges can appear in an extremal configuration but with density $o(1)$. Although sparse, they also force further sharp graphs
as shown in Lemma~\ref{lm:PhamtomSharp}. Similarly it can be shown that they force some further zero
eigenvectors in addition to those given by Lemma~\ref{lm:Eigenvectors}.

This phenomenon was
first observed in~\cite[Section~3.4]{pikhurko+vaughan:Fkl:arxiv}. A new idea
here is that the `test' 3-graph $G$ in the proof of
Lemma~\ref{lm:PhamtomSharp} is not admissible.

The option
\texttt{phantom\_edge} (new in \emph{Flagmatic 2.0}) tells the computer to use these extra identities at the rounding step. 

There happened to be some further zero eigenvectors in addition to those given
by the observations above. Here we just guessed their values by inspecting the floating point
solution and passed the information on to \emph{Flagmatic} using its 
\verb=add_zero_eigenvectors= function.

\subsection{Stability}
In this section we prove Theorem~\ref{f32stability}. Let $G$ be an arbitrary $F_{3,2}$-free 3-graph on $[n]$ with minimum codegree 
$(1/3+o(1))n$. We shall use the information from our flag algebraic proof of Theorem~\ref{f32density} to establish that $G$ lies within edit distance $o(n^3)$ of a balanced $T_{A,B,C}$ construction. First, let us show that almost all $6$-vertex subgraphs of $G$ are sharp $3$-graphs.

\begin{lemma}\label{lm:NonSharpInG} If a $6$-vertex $3$-graph $G_i$ is not sharp, then
$P(G_i,G)=o(n^6)$. 
\end{lemma}
 \begin{proof} 
  Since $\delta_2(G)= n/3 +o(n)$, we have $\sigma_1(G) \geq -o(n^6)$. We know that $\sigma_2(G)\geq 0$ and that $\alpha_j\leq 0$ for all $j \in [426]$. Equality (\ref{eq:a+b}) thus implies that 
 $-o(n^6)\leq \alpha_i P(G_i,G)$. Since $G_i$ is not sharp we have $\alpha_i <0$, from which we deduce that $P(G_i,G)=o(n^6)$ as claimed.
 \end{proof}
By applying a version of an Induced Removal Lemma (see \cite{rodl+schacht:09}
for a very strong version as well as a historical account), we can therefore change
$o(n^3)$ edges of $G$ and destroy all induced copies of non-sharp $3$-graphs,
without creating a copy of $F_{3,2}$. Let $G'$ denote the $3$-graph thus obtained; by definition, all of the $6$-vertex subgraphs of $G'$ are sharp $3$-graphs.

Now, the transcript of our flag algebraic proof of Theorem~\ref{f32density} shows that the number of sharp $3$-graphs and the number of
$6$-vertex $3$-graphs that embed into $T_{A,B,C}$ plus a tripartite $3$-edge are both $13$. 
By Lemma~\ref{lm:PhamtomSharp}, these two families of $6$-vertex $3$-graphs must therefore coincide. In fact, it is routine to
check by hand that there are nine $6$-vertex $3$-graphs that can appear in
$T_{A,B,C}$ as induced subgraphs and that by adding one tripartite $3$-edge to
$T_{A,B,C}$ we increase this number by four.

We deduce from this the following:
\begin{lemma}\label{lm:S} Every $6$-vertex set $X\subseteq V(G')$ admits a
partition $X=A\cup B\cup C$ such that $G'[X]$ is $T_{A,B,C}$ with at most one
tripartite 3-edge added.\qed\end{lemma}

By removing $o(n^3)$ edges from $G$, we may have destroyed our minimum codegree condition, but it will still hold on average: at most $o(n^2)$ pairs can have codegree less than
$(1/3+o(1))n$ in $G'$.

Let us now consider the type $\tau_6$ which is a labelling of $K_4^-$. 
\begin{lemma}\label{lm:1} $P(K_4^-,G')=\Omega(n^4)$.\end{lemma}

\begin{proof} The $3$-graph $G'$ contains at least $\left(\frac{1}{3}+o(1)\right)\binom{n}{3}$ $3$-edges, while it is known that $\pi(K_4^-)<\frac{1}{3}$, as shown by
Matthias~\cite{matthias:94} and Mubayi~\cite{mubayi:03} (the current best known upper-bound is
$\pi(K_4^-)\le 0.2871$, proved by Baber and Talbot~\cite{BaberTalbot11} using
flag algebras). Our claim is thus immediate from the Removal Lemma, or
from supersaturation (see Erd{\H o}s and Simonovits~\cite{ErdosSimonovits83}).\end{proof}

For every quadruple of vertices $abcd$ that induce $K_4^-$ in $G'$ (with
$abc,abd,acd\in E(G')$) form the vector $\B p=\B p_{abcd}$ as in (\ref{eq:x}). The
transcript shows that there are $24$ $\tau_6$-flags with $5$ vertices; thus
$\B p_{abcd}\in \I R^{24}$. Also, the transcript shows that the rank of
$Q=Q^{\tau_6}$ is $23$; thus the nullspace of $Q$ is
1-dimensional. From Lemma~\ref{lm:Eigenvectors} we know
that the (unique up to a scaling) forced zero eigenvector $\B z$ of $Q$ 
consists of $21$ entries $0$ and three equal entries that correspond to the
three $\tau_6$-flags with the unlabelled vertex having the following
links in $abcd$: 1) $ab,ac,ad$ 2) $bc,bd,cd$ 3) empty. Indeed, the only way we
see $\tau_6$ in $T_{V_1,V_2,V_3}$ is when $a\in V_i$ and $b,c,d\in V_{i-1}$ for
some $i\in \I Z_3$; by choosing the unlabelled vertex  $x$ in
respectively $V_{i-1}$,  $V_{i}$, $V_{i+1}$, we get these link graphs (each
appearing about $n/3$ times when each $|V_j|=n/3$). Scale $\B z$ so that it has unit $\ell_2$-norm
$\|\B z\|=1$. 

Take a spectral decomposition $Q=\sum_{i=1}^{23}\lambda_i \B
f_i^T\B f_i$, where the $\B f_i$ are eigenvectors of $Q$
such that $\{\B f_1,\dots,\B f_{23},\B z\}$
forms an orthonormal basis of $\mathbb{R}^{24}$. Since $Q\succeq 0$ has rank $23$, we have that each $\lambda_i>0$.
Let $\lambda=\min(\lambda_1,\dots,\lambda_{23})>0$, a positive constant
independent of $n$. Since
$(\B p,\B p)=(\B p,\B z)^2+\sum_{i=1}^{23} (\B p,\B f_i)^2$, we have
 \begin{equation}\label{eq:lambda}
 \B pQ\B p^T=\sum_{i=1}^{23} \lambda_i (\B p, \B f_i)^2\ge \lambda ((\B p,\B p)-(\B p,\B
z)^2). 
 \end{equation}
 Note that for all $abcd$ inducing $\tau_6$, we have $\|\B p_{abcd} \|^2=\Omega(n^2)$.
We know that $\sum_{abcd} \B p_{abcd}Q\B p_{abcd}^T
 =O(n^5)$.  Thus,
by Lemma~\ref{lm:1},
the right-hand side of (\ref{eq:lambda}) is 
$O(n)=o(\|\B p_{abcd}\|^2)$ for
all but $o(n^4)$ quadruples $abcd$ inducing $\tau_6$. Fix one such `typical' quadruple $abcd$ and
consider $\B p=\B p_{abcd}$.
By the cosine formula, the approximate equality \[(\B p,\B
z)^2= (\B p,\B p)+O(n)=\|\B p\|^2\|\B z\|^2 (1+o(1))\] implies that $\B p$ and $\B z$ are almost collinear.
It follows
that $\B p\in \I R^{24}$ has $21$ coordinates with values $o(n)$ and $3$ coordinates taking values
$(1/3+o(1))n$ corresponding to the $\tau_6$-flags 1)--3) defined above. 
So, if we define
 \begin{eqnarray*}
 V_1&=&\{x\in V(G')\mid {G'}_x[abcd]=\{ab,ac,ad\}\}\\
 V_2&=& \{x\in V(G')\mid {G'}_x[abcd]=\{bc,bd,cd\}\},\\
 V_3&=& \{x\in V(G')\mid {G'}_x[abcd]=\emptyset\},
 \end{eqnarray*}
 then for each $i\in[3]$ we have $\vert V_i\vert= (1/3+o(1))n$. Let $W= [n] \setminus \bigcup_{i=1}^3 V_i$. Since $\vert W \vert =o(n)$, it is sufficient to show that the induced subgraph $G'[\bigcup_{i=1}^3 V_i]$ lies within edit distance $o(n^3)$ of the $3$-graph $T_{V_1, V_2, V_3}$ to conclude our proof of Theorem~\ref{f32stability}. We shall do this via a succession of easy lemmas. We again use `$x_1x_2\vert y_1y_2y_3$' as a notational shorthand for the statement that the $3$-edges $x_1x_2y_1,x_1x_2y_2,x_1x_2y_3$ and $y_1y_2y_3$ are all present in our graph (and thus that $\{x_1x_2y_1y_2y_3\}$ spans a copy of $F_{3,2}$, contradicting our assumption that $G'$ is $F_{3,2}$-free).

\begin{lemma}\label{cl:1} $G'[V_1]$ and $G'[V_2]$ are empty $3$-graphs.\end{lemma}
\begin{proof} Indeed, if $xyz\in G'[V_1]$, then $ab\vert xyz$, while if $xyz\in G'[V_2]$,
then $bc\vert xyz$, both of which are contradictions.\end{proof}

\begin{lemma}\label{cl:2} $G'$ has no $3$-edges of the form $V_1V_2V_2$, that is, $3$-edges with two vertices in $V_2$ and one in $V_1$.\end{lemma}

\begin{proof} Take any $z\in V_1$ and distinct $x,y\in V_2$. Consider $G'[abcdxz]$. By
Lemma~\ref{lm:S}, we have that $G'[abcdxz]=T_{A,B,C}$ plus at most one
tripartite edge for some partition $abcdxz=A\cup B\cup C$.
Since $G'[abcd]\cong K_4^-$, it follows that $bcd$ are in one part, say $A$, and
$a$ lies in the next part $B$. Since $xbc,xbd,xcd\in E(G')$, we must have $x\in B$.
Likewise $z\in A$. Thus necessarily $xzb,xzc,xzd\in E(G')$. 

Likewise $yzb,yzc,yzd\in E(G')$. So if $xyz\in E(G')$ also, then $zy\vert bdx$, a contradiction.\end{proof}

\begin{lemma}\label{cl:3} All but $o(n^3)$ $3$-edges of the form $V_2V_2V_3$ are in
$G'$.\end{lemma}

\begin{proof} By our observation that most (all but $o(n^2)$) pairs in $G'$ have codegree at least $(1+o(1))n/3$, by the fact that $\vert W \vert=o(n)$ and by Lemma~\ref{cl:1}, the 3-graph $G'[\bigcup_{i=1}^3 V_i]$ must have at least $(1-o(1)){n/3\choose 2}\times
n/3$ 3-edges that intersect the independent set $V_2$ in at least two vertices.
By Lemma~\ref{cl:2}, all these 3-edges are of the form $V_2V_2V_3$, giving the
required result.\end{proof}

\begin{lemma}\label{cl:4} $V_3$ spans $o(n^3)$ 3-edges in $G'$.\end{lemma}

\begin{proof} By Lemma~\ref{cl:3}, for all but $o(n^2)$ $x,y \in V_2$
we have that $\vert V_3\setminus \Gamma(x,y) \vert=o(n)$. But $\Gamma(x,y)$ is an
independent set as $G'$ is $F_{3,2}$-free. The lemma follows.\end{proof}
Let $i \in \{1,2,3\}$. We write $V_{i+1}$ for the part coming after $V_i$ in the cyclic order on $\{1,2,3\}$, so that $V_{3+1}=V_1$, $V_{1-1}=V_3$, etc.
\begin{lemma}\label{cl:5} If all but $o(n^3)$ 3-edges $V_iV_iV_{i+1}$ are in
$G'$, then all but $o(n^3)$ 3-edges $V_iV_{i+1}V_{i+1}$ are not in $G'$.
\end{lemma}

\begin{proof} 
By the assumption of the lemma, for all but $o(n^5)$ $5$-tuples of vertices $z,z',z''\in V_i$ and $x,y\in V_{i+1}$, we have $xzz',xzz'',yz'z''\in E(G')$. To prevent $xz \vert yz'z''$, we must have $xyz\not\in E(G')$.\end{proof}

By Lemmas~\ref{cl:3} and~\ref{cl:5} we conclude that all but at most $o(n^3)$
$3$-edges of the form $V_2V_3V_3$ are not in $E(G')$. This together with Lemma~\ref{cl:4} implies that
almost all $3$-edges of the form $V_3V_3V_1$ are in $G'$ in the same way as we showed that almost all $V_2V_2 V_3$ $3$-edges are in $G'$ in Lemma~\ref{cl:3}. Now, by Lemma~\ref{cl:5} again, we have that only $o(n^3)$ $3$-edges of the form $V_1V_1V_3$ belong to $E(G')$.

Finally, to finish the proof of stability,
it remains that at most $o(n^3)$ 3-edges are of the form $V_1V_2V_3$. For all but $o(n^5)$ $5$-tuples $x,x'\in V_{1}$, $y\in V_2$, and $z,z'\in V_{3}$, we have
$xx'y,x'zz'\in E(G')$. Thus at least one of $xyz,xyz'$ is missing from $G'$ 
(to prevent $xy \vert x'zz'$).
However, if we had $\Omega(n^3)$ 3-edges of the form $V_1V_2V_3$, then we would have $\Omega(n^4)$ choices of $x,y,z,z'$ with both $xyz,xyz'$ being
in $E(G')$, a contradiction.

It follows that $G'$ (and hence $G$) lies within edit distance $o(n^3)$ of a balanced $T_{V_1, V_2, V_3}$ configuration. This concludes the proof of Theorem~\ref{f32stability}. \qed

\section{The codegree threshold}\label{exactsection}
In this section, we determine the codegree threshold of $F_{3,2}$ for all
sufficiently large $n$. This is a simple (but long) chain of arguments from
stability, with a slight twist at the end when we deal with the fact that the
extremal constructions are not unique and depend on the congruence class of $n$
modulo~$3$.

We know from Theorem~\ref{f32stability} that almost extremal $3$-graphs are
close to balanced $T_{A,B,C}$ constructions. We use this fact as our starting
point and analyse an extremal example $G$ via a series of lemmas to show
that in fact $G$ is not only close to a certain fixed, balanced $T_{A,B,C}$
construction, but that it consists exactly of a subgraph of this $T_{A,B,C}$
construction together with a small number of `tripartite' $3$-edges. As an
immediate corollary, we have that for all $n$ sufficiently large,
$\mathrm{coex}(n, F_{3,2})\leq \lfloor n/3 \rfloor$.

At that point we separate into cases corresponding to the congruence class of
$n$ modulo $3$, and determine both the codegree threshold and the extremal
constructions for all $n$ sufficiently large. 

\subsection {The structure of almost extremal configurations}\label{stabilityanalysis}
In our argument, we shall frequently need to locate potential
$F_{3,2}$-subgraphs inside larger $3$-graphs, and it will be convenient just as
in Sections~\ref{edssection} and~\ref{flagsection} to write $ab \vert cde$ to mean that $abc, abd, abe$
and $cde$ are all $3$-edges (and thus that $\{abcde\}$ spans a copy of
$F_{3,2}$).

Let $G$ be a $3$-graph on $n$ vertices with independent neighbourhoods and
minimal codegree $\delta_2(G)\geq n/3+o(n)$. Pick a partition of its vertex set
$V(G)= V_1 \cup V_2 \cup V_3$ such that $\vert E(G) \setminus E(T_{V_1,V_2,
V_3}) \vert$ is minimised.

Write $T$ for $T_{V_1,V_2, V_3}$. Set $B= E(G) \setminus E(T)$ to be the set of
\emph{bad} $3$-edges, i.e.\ $3$-edges which are in $G$ and not
in $T$, and set $M=E(T)\setminus E(G)$ to be the set of \emph{missing}
$3$-edges, i.e.\ $3$-edges which are in $T$ but not in $G$.

By Theorem~\ref{f32stability}, we know that $G$ lies at edit distance $o(n^3)$
of a balanced $T_{A,B,C}$ construction. 
As an easy consequence of this fact, we have the following:

\begin{lemma}\label{Talmostbalanced}
\begin{enumerate}[(i)]
\item $\vert B \vert = o(n^3)$,
\item $\vert M \vert = o(n^3)$,
\item $\vert V_i \vert = n/3 +o(n)$ for $i=1,2,3$.
\end{enumerate}
\end{lemma}
\begin{proof}
Since the edit distance between $G$ and a balanced $T_{A,B,C}$ construction is
$o(n^3)$, we have that $\vert B \vert =o(n^{3})$. (Since otherwise $T$ would not
be minimising $\vert E(G) \setminus E(T) \vert$.)

Let $\alpha_i=\vert V_i \vert/n$ for $i=1,2,3$. The
number of $3$-edges in $G$ with at least two vertices in $V_i$ is at most the
number of $3$-edges in $T$ with this property plus the total number of bad
$3$-edges $\vert B \vert$. In particular the average codegree in $G$ of pairs of
vertices in $V_i$ is at most
\[\left({\alpha_i}^2\alpha_{i+1} n^3/2
+o(n^3)\right)/\left({\alpha_i}^2n^2/2\right)= \alpha_{i+1}n+o(n).\]
Since $\delta_2(G) \geq n/3+o(n)$, we must have in particular
$\alpha_i=1/3+o(1)$ for $i=1,2,3$. We have thus established parts (i) and (iii)
of our lemma.

Finally for part (ii) observe that the total number of $3$-edges in $G$
satisfies
 $$
e(G) = \sum_{x,y \in V(G)} \frac{d(x,y)}3
 \geq \binom{n}{2}\, \frac{\delta_2(G)}3
= \frac{n^3}{18}+o(n^3).
 $$
It then follows from (iii) and (i) that $\vert M \vert=\vert E(T)\vert-\vert E(G)\vert +\vert B
\vert$ is $o(n^3)$.
\end{proof}

Now let us analyse the \emph{link graphs} of vertices in $G$. Given $x \in
V(G)$, let $G_x$ be the $2$-graph on $V(G)$ with $2$-edges $\{uv: \ xuv \in
E(G)\}$ and let $e(G_x)= \vert E(G_x)\vert$ be the number of edges it contains.
Also let $G_x[V_i]$ denote the subgraph of $G_x$ induced by the vertices in
$V_i$, 
\[G_x[V_i]=(V_i, \{uv \in E(G_x):\ u,v \in V_i\})\] 
and let $G_x[V_i, V_j]$ denote the bipartite subgraph of $G_x$ on $V_i \cup V_j$
with edges $\{uv\in E(G_x): \ u \in V_i, v \in V_j\}$.

 We shall also
write $V_{i+1}$ for the part coming after $V_i$ in the cyclic order on
$\{1,2,3\}$, so that $V_{3+1}=V_1$.

We first prove six lemmas which show that the link graphs of \emph{all}
vertices of $G$ look like they ought to (up to some small error) if $G$ was a
$T_{A,B,C}$ construction.
\begin{lemma}\label{linkatmostonevi}
For every $x \in V(G)$, there is at most one $i\in\{1,2,3\}$ for which
$e(G_x[V_i])= \Omega(n^2)$.
\end{lemma}
\begin{proof}
Pick $x \in V(G)$, and suppose that both $V_1$ and $V_2$ contain $\Omega(n^2)$
edges of $G_x$. Then there are $\Omega(n^4)$ choices of pairs $yz \in
E(G_x[V_1])$ and $vw \in E(G_x[V_2])$. For each such choice, at least one of the
triples $yzv$ and $yzw$ is
missing from $G$ and lies in $M$ (for otherwise we would have $yz\vert vwx$,
violating the assumption that $G$ is $F_{3,2}$-free).

Now each such forbidden triple is counted in at most $n$ quadruples $\{v,w, y,
z\}$, implying that $\vert M \vert= \Omega(n^3)$, and contradicting part (ii) of
Lemma~\ref{Talmostbalanced}.
\end{proof}

\begin{lemma}\label{linkdiffparts}
For every $x \in V(G)$ , there are at most $o(n^3)$ triples $w,y,z$ such that
$wz, yz \in E(G_x)$ and $w,y$ come from two different parts $V_i$, $i \in
\{1,2,3\}$.
\end{lemma}
\begin{proof}
Pick $x \in V(G)$ and suppose for contradiction that $\Omega(n^3)$ such triples
could be found. Then in particular we can find $\Omega(n^4)$ quadruples
$v,w,y,z$ such that $vz, wz$ and $yz$ all lie in $E(G_x)$ and $y \in V_{i}$,
$v,w \in V_{i-1}$ for some $i\in \{1,2,3\}$.

For each such quadruple, the triple $vwy$ is missing from $G$ and lies in $M$
(for otherwise we would have $xz \vert vwy$). As before, each such triple is
counted in at most $n$ quadruples, giving $\vert M\vert = \Omega(n^3)$ missing
edges and contradicting part (ii) of Lemma~\ref{Talmostbalanced}.
\end{proof}

\begin{lemma}\label{linkonlyonevi}
For every $x \in V(G)$, exactly one of $V_1$, $V_2$, $V_3$ contains
$\Omega(n^2)$ $2$-edges of $G_x$.
\end{lemma}
\begin{proof}
Pick $x \in V(G)$. By Lemma~\ref{linkatmostonevi}, we know that at most one of
$e(G_x[V_1])$, $e(G_x[V_2])$ and $e(G_x[V_3])$ may be of order $\Omega(n^2)$.
Assume for contradiction that all three are of order $o(n^2)$. Then for every
$i$, all but $o(n)$ vertices in $V_i$ have $o(n)$ neighbours in $G_x[V_i]$.

Lemma~\ref{linkdiffparts} implies that for all but $o(n)$ vertices $z\in V_i$ at least one of $\Gamma(x,z)\cap V_{i+1}$, $\Gamma(x,z)\cap V_{i-1}$ has size $o(n)$.
Thus we can partition all but $o(n)$ vertices of $V_i$ into two parts $V_i'$ and $V_i''$ satisfying the
following:
\begin{itemize}
\item for every $z \in V_i'$, there are at most $o(n)$ $y \in V_i\cup V_{i+1}$ such
that $yz \in E(G_x)$;
\item for every $z \in V_i''$, there are at most $o(n)$ $y \in V_{i-1}\cup V_i$
such that $yz \in E(G_x)$.
\end{itemize}
Since for every $z \in V(G)$ the codegree of $x$ and $z$ in $G$ is at least
$n/3+o(n)$, since by Lemma~\ref{Talmostbalanced} we have $\vert V_i \vert= n/3
+o(n)$ for $i=1,2, 3$, and since $e(G_x[V_i])=o(n^2)$ by assumption, it follows
that for every $i$ the following hold:
\begin{itemize}
\item $G_x[V_{i-1}, V_i']$ is almost complete bipartite (contains all but
$o(n^2)$ of the possible $2$-edges);
\item $G_x[V_i'', V_{i+1}]$ is almost complete bipartite (contains all but
$o(n^2)$ of the possible $2$-edges).
\end{itemize}
Now if $V_1'$ contained $\Omega(n)$ vertices then almost all vertices in $V_{3}$
send $\Omega(n)$ edges to $V_1'\subseteq V_1$. If follows in particular that
$\vert V_{3}'\vert =o(n)$.  Similarly, if $V_1''$ contained $\Omega(n)$ vertices
then it would follow that $\vert V_{2}''\vert=o(n)$.

Thus if both $V_1'$ and $V_1''$ contained $\Omega(n)$ vertices, then there would
be only $o(n^2)$ edges of $G_x$ between $V_{2}$ and $V_{3}$. Since we are also
assuming that $V_{3}$ contains only $o(n^2)$ edges of $G_x$, it follows that the
average degree in $G_x$ of vertices in $V_{3}$ is at most $\vert V_1' \vert
+o(n)$. But now since $\vert V_1 \vert=n/3+o(n)$, and since $V_1'$ and $V_1''$
are disjoint subsets of $V_1$ both containing $\Omega(n)$ vertices, it follows
that this average degree is at most $(1-c)n/3+o(n)$ for some strictly positive
constant $c>0$. For $n$ sufficiently large, this contradicts the fact that the
minimal codegree in $G$ is at least $n/3+o(n)$ (since the degree of a vertex in $G_x$ is its codegree with $x$ in $G$).

On the other hand if we had, for example, $\vert V_1'\vert = \vert V_1 \vert
+o(n)$ then all but $o(n)$ vertices from $V_3$ send $\Omega(n)$ edges to $V_1$
in $G_x$, so that $\vert V_3 \vert = \vert V_3'' \vert +o(n)$. But now by
definition of $V_1'$ and $V_3''$, there are only $o(n^2)$ edges of $G_x$ from
$V_1 \cup V_3$ to $V_2$. Since we are assuming that $e(G_x[V_2])=o(n^2)$ this
implies in particular that all but $o(n)$ vertices in $V_2$ have degree $o(n)$
in $G_x$, which again contradicts the fact that $\delta_2(G) \geq n/3+o(n)$.
\end{proof}
\begin{lemma}\label{linkviorvivi+1}
For every $x \in V(G)$ and every $i\in\{1,2,3\}$ we have $e(G_x[V_i])=o(n^2)$ or
$e(G_x[V_i, V_{i+1}])=o(n^2)$.
\end{lemma}
\begin{proof}
Pick $x \in V(G)$ and suppose the claim of the lemma does not hold for some $i$.
Then we have $\Omega(n^4)$ possible choices of a quadruple $\{v,w,y,z\}$ with
$vw \in E(G_x[V_i])$ and $yz \in E(G_x[V_i, V_{i+1}])$. For each such choice, at
least one of the triples $vyz$, $wyz$ is missing from $G$ and lies in $M$ (for
otherwise we would have $yz\vert vwx$).

Each such forbidden triple is counted in at most $n$ quadruples, so, just as in
Lemmas~\ref{linkatmostonevi} and~\ref{linkdiffparts}, this implies $\vert M
\vert = \Omega(n^3)$, contradicting Lemma~\ref{Talmostbalanced} part (ii).
\end{proof}

With these lemmas in hand, we can now show that $G$ has no vertex of high
\emph{bad} or \emph{missing} degree, where the bad degree $d_B(x)$ is just the
number of bad $3$-edges incident with $x$ while the missing degree $d_M(x)$ is
the number of $3$-edges from $M$ incident with $x$. 

\begin{lemma}\label{nobigbaddegree}
For every $x \in V(G)$, $d_B(x)=o(n^2)$.
\end{lemma}
\begin{proof}
Pick $x \in V(G)$. By Lemma~\ref{linkonlyonevi}, we may assume without loss of
generality that $e(G_x[V_1])$ and $e(G_x[V_2])$ are both $o(n^2)$, while
$e(G_x[V_3])=\Omega(n^2)$, just as would expect it to be if $G$ was a subgraph
of $T_{V_1,V_2,V_3}$ and $x$ was chosen from $V_1$.

By Lemma~\ref{linkviorvivi+1}, we then know that $e(G_x[V_3, V_1])=o(n^2)$. Thus
for $y \in V_1$ there are on average only $o(n)$ edges of $G_x$ joining $y$ to
vertices in $V_1\cup V_3$. On the other hand we know from the codegree condition
on $G$ that for every $y \in V_1$ the joint neighbourhood of $x$ and $y$ has
size at least $n/3+o(n)$. Since $\vert V_2\vert=n/3+o(n)$
(Lemma~\ref{Talmostbalanced}, part (iii)), it follows that for all but $o(n)$ vertices $y \in
V_1$, $y$ is adjacent in $G_x$ to all but at most $o(n)$ vertices $z \in V_2$. In particular,
$G_x[V_1,V_2]$ is almost complete: at most $o(n^2)$ of the possible edges
between $V_1$ and $V_2$ are missing.

This and Lemma~\ref{linkdiffparts} imply that $e(G_x[V_2,V_3])=o(n^2)$. Thus all but $o(n^2)$
edges of $G_x$ are internal to $V_3$ or lie between $V_1$ and $V_2$. If $x \in
V_1$ then $d_B(x)=o(n^2)$, whereas if $x \in V_2\cup V_3$, we would have $d_B(x)
=\Omega(n^2)$. Since our partition $V_1\cup V_2 \cup V_3$ was chosen to minimise
the number of bad $3$-edges, it must be that $x$ was assigned to $V_1$. The
claim of the lemma thus holds for~$x$.
\end{proof}

\begin{lemma}\label{nobigmissingdegree}
For every $x \in V(G)$, $d_M(x)=o(n^2)$
\end{lemma}
\begin{proof}
Pick $x \in V(G)$, and write $d_T(x)$ for the number of $3$-edges of
$T=T_{V_1,V_2,V_3}$ containing $x$. Since by Lemma~\ref{Talmostbalanced} we have
$\vert V_i \vert = n/3+o(n)$ for $i=1,2,3$, it readily follows that $d_T(x)=
n^2/6+o(n^2)$.

Now the codegree condition $\delta_2(G)\geq n/3+o(n)$ tells us that every $y \in
V(G) \setminus\{x\}$ is incident with at least $n/3+o(n)$ edges in $G_x$. It
follows in particular that
\[e(G_x) = \frac{1}{2}\sum_y d(x,y)\geq \frac{n^2}{6}+o(n^2).\]
Thus
 $$
d_M(x)=d_B(x)+d_T(x)- e(G_x)\leq d_B(x)+o(n^2),
 $$
which by Lemma~\ref{nobigbaddegree} is $o(n^2)$, as desired.
\end{proof}

We can now show that in fact all bad edges are \emph{tripartite}, i.e.\ meet
each of $V_1$, $V_2$ and $V_3$ in one vertex.
\begin{lemma}\label{nointernalbad}
For every $i \in \{1,2,3\}$, $V_i$ is an independent set in $G$.
\end{lemma}
\begin{proof}
Suppose for contradiction that we had a $3$-edge of $G$ entirely contained
within $V_i$ for some $i$. Without loss of generality, we may assume that we
have $\{x,y,z\} \in E(G)$ with all of $x,y,z$ lying in $V_1$. Then for every
pair $u,v$ from $V_3$, we have that at least one of the triples $uvx$, $uvy$,
$uvz$ is missing from $G$, for otherwise $uv \vert xyz$. There are $n^2/18
+o(n)$ such pairs $uv$ (since $\vert V_3 \vert =n/3+o(n)$). It follows that at
least one of $\{x,y,z\}$ has missing degree at least $n^2/54+o(n)$. This
contradicts Lemma~\ref{nobigmissingdegree}.
\end{proof}
\begin{lemma}\label{nobipartitebad}
For every $i \in \{1,2,3\}$, there are no $3$-edges with two vertices in $V_i$
and one in $V_{i-1}$.
\end{lemma}
\begin{proof}
Suppose we had such a bad $3$ edge -- without loss of generality $xyz \in E(G)$
with $x,y \in V_3$ and $z \in V_2$. Since $\delta_2(G)\geq n/3+o(n)$, the joint neighbourhood $\Gamma(x,y)$ contains at least $n/3+o(n)$
vertices. We know from Lemma~\ref{nointernalbad} that $\Gamma(x,y) \subseteq V_1
\cup V_2$.

Suppose $\vert \Gamma(x,y) \cap V_1\vert=\Omega(n)$. Then there are
$\Omega(n^2)$ $a,a' \in V_1$ such that $axy$ and $a'xy$ are both in $E(G)$. But
for such pairs, the $3$-edge $aa'z$ is missing from $G$, since otherwise we
would have $xy \vert aa'z$. It follows that $d_M(z)=\Omega(n^2)$, contradicting
Lemma~\ref{nobigmissingdegree}.

We must therefore have $\vert \Gamma(x,y) \cap V_1 \vert =o(n)$ and thus by the
codegree condition $\vert \Gamma(x,y) \cap V_2 \vert = n/3+o(n)$. Now, consider
triples $w,w', w''$ from $V_2$. For all but $o(n^3)$ triples, $xyw$ is in
$E(G)$. Also, since $d_M(x)=o(n^2)$ by Lemma~\ref{nobigmissingdegree}, for all
but $o(n^3)$ of such triples, both of $xww'$ and $xww''$ are in $E(G)$. But then
$w'w''y$ is missing from $G$, as otherwise we would have $xw\vert yw'w''$. This
implies that $d_M(y)=\Omega(n^2)$, contradicting Lemma~\ref{nobigmissingdegree}.

It follows that we cannot have bad $3$-edges taking one vertex in $V_{i-1}$ and
two vertices in $V_{i}$.
\end{proof}

\begin{corollary}\label{upperboundcoex}
\[\delta_2(G)\leq \lfloor n/3 \rfloor.\]
\end{corollary}
\begin{proof}
Suppose without loss of generality that $V_1$ is the smallest of the three parts
$V_1$, $V_2$ and $V_3$. Then $\vert V_1 \vert \leq \lfloor n/3 \rfloor$. Now
consider a pair of vertices $x,y \in V_3$. By Lemmas~\ref{nointernalbad}
and~\ref{nobipartitebad}, there is no bad edge of $G$ containing both $x$ and
$y$. In particular the codegree of $x$ and $y$ in $G$ is at most the codegree of
$x$ and $y$ in $T$, which is exactly $\vert V_1 \vert $.
\end{proof}

\subsection{Divisibility and tripartite matchings}

By Corollary~\ref{upperboundcoex}, we know that for $n$ large enough
$\mathrm{coex}(n, F_{3,2})\leq \lfloor n/3 \rfloor$.
Construction~\ref{orcyclebip} from the Introduction shows that for all $n$ we
have $\mathrm{coex}(n, F_{3,2})\geq \lfloor n/3 \rfloor -1$. Continuing on the
work in the previous section (and re-using the previous section's notation), we
now determine for $n$ large enough which of the two possible values is the
actual codegree threshold. In addition, we seek to describe the set of
extremal examples. As this set depends on some divisibility conditions ---
specifically, on the congruence class of $n$ modulo $3$ --- we separate out into
three cases.

Before we do so, however, let us introduce some useful terminology. Let
$V_1\sqcup V_2 \sqcup V_3$ be a tripartition of a vertex set $V$. A
\emph{tripartite} $3$-edge is a triple $x_1x_2x_3$ with $x_i \in V_i$ for
$i=1,2,3$. Let $F$ be a set of tripartite $3$-edges. A pair
of vertices is \emph{overused (by $F$)} if it is contained in 
at least two $3$-edges of $F$. Next, $F$ is a \emph{tripartite pair
matching}, or just a \emph{tripartite matching}, if every two elements
of $F$ intersect in at most one vertex (that is,
there are no overused pairs).

\begin{proposition}\label{matchingsaregood}
Let $V$ be a set of vertices with tripartition $V=V_1 \sqcup V_2 \sqcup V_3$.
Then for any tripartite pair matching $F$ the $3$-graph $G$ on $V$ obtained by
adding the $3$-edges in $F$ to $T_{V_1, V_2, V_3}$ is $F_{3,2}$-free. 
\end{proposition}
\begin{proof}
This is a simple check. We know that $T_{V_1, V_2, V_3}$ is $F_{3,2}$-free. By
symmetry of the construction, it is sufficient to check that for
every
$a,a', a''' \in V_1$, $b,b' \in V_2$ and $c \in V_3$, neither of the $5$-sets
$\{a,a',b,b',c\}$ and $\{a,a', a'', b, c\}$ induce a copy of $F_{3,2}$ in $G$.
Without loss of generality the $3$-edges contained in these two $5$-sets are
subsets of $\{aa'b, aa'b',bb'c, abc,a'b'c\}$ and $\{aa'b,aa''b, a'a''b, abc\}$
respectively, neither of which contains a copy of $F_{3,2}$.
\end{proof}

\subsubsection{The case $n$ congruent to $0$ modulo $3$}

When $n$ is congruent to $0$ modulo $3$ and sufficiently large, the upper-bound
in Corollary~\ref{upperboundcoex} is sharp, and moreover there is a simple
description of all extremal configurations.

Before we give this construction, let us recall a basic fact from graph theory.
A \emph{proper edge colouring} of a $2$-graph $G$ with $m$ colours is a map
$\phi$ which assigns to each edge $\{a,b\}\in E(G)$ a \emph{colour}
$\phi(a,b)\in [m]$, such that edges which meet at a vertex are assigned
different colours. It is trivial to check that if $G$ is the complete bipartite
$2$-graph $K_{m,m}=\left([2m],\{ij:\ i \in [m], j \in [2m]
\setminus[m]\}\right)$ then there exists a proper edge colouring of $G$ with $m$
colours. (Consider e.g.\ $\phi(i,j)= i+j \ (\textrm{mod}\ m)$.) Such
edge colourings are in bijective correspondence with Latin squares. 
We do not have an explicit description of all such structures; 
in fact, even the counting problem is difficult (see
e.g.~\cite{mckay+wanless:05}).

\begin{construction}[Family $\CT(3m)$]\label{co:CT}
Let $n=3m$. Take disjoint sets $A,B,C$, each of size $m$.
Assume, for convenience, that $C=[m]$. 
Let $\phi$ be an edge colouring of the complete bipartite $2$-graph
with parts $A$ and $B$ with $m$ colours. Take the $3$-graph $T_{A,B,C}$
and all triples $abc$ where $a\in A$, $b\in B$ and $\phi(ab)=c$.
\end{construction}

It follows from the definition of proper colourings that $F$ is a tripartite
pair matching on $A\sqcup B\sqcup C$. Thus every $H\in\CT(n)$ is
$F_{3,2}$-free
by
Proposition~\ref{matchingsaregood}. Furthermore, all vertex pairs in $H$
have codegree $m$. It follows from Corollary~\ref{upperboundcoex} that $H$ is
extremal for the
codegree problem for all $n$ sufficiently large.

\begin{corollary}\label{thresholdn0mod3}
For all $n$ divisible by $3$ and sufficiently large, $\mathrm{coex}(n,
F_{3,2})=n/3$.\qed
\end{corollary}

What is more, every extremal configuration belongs to $\CT(n)$.

\begin{theorem}\label{n=0mod3config} Let $n=3m$ be large. 
Let $G$ be an $F_{3,2}$-free $3$-graph such that $v(G)=n$ and
$\delta_2(G)=m$. Then $G\in\CT(n)$.
\end{theorem}
\begin{proof}
Let $V_1$, $V_2$ and $V_3$ be as in
Section~\ref{stabilityanalysis}. Consider any pair of
vertices from $V_1$. By Lemmas~\ref{nointernalbad} and~\ref{nobipartitebad},
their joint neighbourhood is a subset of $V_2$, so that by the codegree
condition we must have $\vert V_2\vert \geq m$. Similarly we have $\vert
V_3\vert$ and $\vert V_1 \vert$ both at least $m$, so that in fact we must have
$\vert V_i \vert =m$ for $i=1,2,3$. Furthermore, observe that all $3$-edges
taking two vertices $x,x'$ in $V_i$ and one in $V_{i+1}$ must be in $E(G)$
(otherwise the pair $x,x'$ would have codegree at most $m-1$). So there are no missing edges in $G$.

Write $F$ for the set of tripartite $3$-edges of $G$ associated with the
partition $V_1 \sqcup V_2 \sqcup V_3$. We claim that $F$ contains no overused
pair. Indeed suppose this was not the case. Without loss of generality we would
then have vertices $a \in V_1$, $b \in V_2$ and $c,c'$ in $V_3$ such that $abc$
and $abc'$ are both in $F$ and hence in $G$. Now let $a'$ be any vertex in $V_1
\setminus\{a\}$. By the observation in the previous paragraph, both of $cc'a'$
and $aa'b$ are in $E(G)$. But then we would have $ab \vert cc'a'$, a
contradiction.

Now let $b \in V_2$ and $c\in V_3$. We know that $\vert \Gamma (b,c)\vert \geq
m$, that $\Gamma(b,c)\subseteq V_1 \cup V_2\setminus\{b\}$
(Lemma~\ref{nobipartitebad}).
Thus there exists at least one vertex $a=\psi_c(b)\in V_1$ with $abc\in
E(G)$, and this vertex is unique (else $(b,c)$ would be an overused pair). What
is more if $b'$ is an element of $V_2$ distinct from $b$, then we cannot have
both of $ab'c$ and $abc$ being $3$-edges of $G$, for otherwise $F$
would have an overused pair $\{a,c\}$. Since there are $m$ distinct elements in
each of $V_1$ and $V_2$, it follows that for any $c \in V_3$, $\psi_c$ is a
bijection from $V_2$ to $V_1$. Finally observe that if $c$ and $c'$ are distinct
elements of $V_3$ then for any $b \in V_2$, $\psi_c(b)\neq \psi_{c'}(b)$, since
otherwise $\{b,\psi_c(b)\}$ would be an overused pair for $F$. In particular the
map $\phi$ assigning colour $c$ to the $2$-edge $(b, \psi_c(b))$ is an edge
colouring of the complete bipartite $2$-graph between $V_1$ and $V_2$ using $m$
colours.

The $3$-graph $G$ thus belongs to $\CT(n)$, as claimed.
\end{proof}

\subsubsection{The case $n$ congruent to $2$ modulo $3$}

When $n$ is congruent to $2$ modulo $3$ and sufficiently large, the upper bound
in Corollary~\ref{upperboundcoex} is again sharp. Extremal constructions are
very similar to the ones in the previous case. However, there are now
some $3$-edges in the extremal configuration which can be deleted without
lowering the minimal codegree, so that a proof of an analogue of
Theorem~\ref{n=0mod3config} becomes more delicate.

\begin{construction}[Family $\CT(3m+2)$] Pick any $H$ from the family
$\CT(3m+3)$ that was defined by Construction~\ref{co:CT} and remove one vertex
from $H$.\end{construction}

Clearly, any obtained $3$-graph is $F_{3,2}$-free and, as it is easy to check,
has minimum codegree $m$.

\begin{corollary}\label{thresholdn2mod3}
For all $n$ congruent to $2$ modulo $3$ and sufficiently large,
$\mathrm{coex}(n, F_{3,2})=\lfloor n/3\rfloor$.~\qed
\end{corollary}

\begin{theorem}\label{n=2mod3config}
Let $n=3m+2$ be large. 
Let $G$ be an $F_{3,2}$-free $3$-graph with $v(G)=n$ and
$\delta_2(G)=m$. Then
$G$ is a subgraph of some $H\in\CT(n)$. 
\end{theorem}
\begin{proof}
Let $V_1$, $V_2$, $V_3$ be as in
Section~\ref{stabilityanalysis}. Consider any pair of
vertices from $V_1$. By Lemmas~\ref{nointernalbad} and~\ref{nobipartitebad},
their joint neighbourhood is a subset of $V_2$, so that by the codegree
condition we must have $\vert V_2\vert \geq m$. Similarly we have $\vert
V_3\vert$ and $\vert V_1 \vert$ both at least $m$.

Without loss of generality, we may therefore assume that $\vert V_3 \vert =m$,
and $m \leq \vert V_i \vert \leq m+2$ for $i=1,2$. We know
(Lemmas~\ref{nointernalbad} and~\ref{nobipartitebad}) that for every $b,b' \in
V_2$ their joint neighbourhood is a subset of $V_3$. By the codegree condition
$\delta_2(G)=m$, it follows that \emph{all} $3$-edges taking two vertices in
$V_2$ and one vertex in $V_3$ must be in $E(G)$. We claim that in addition all
$3$-edges taking two vertices in $V_3$ and one in $V_1$ must be in $E(G)$:
\begin{lemma}\label{n2mod3nomissingbipartite}
For all $c,c' \in V_3$ and all $a \in V_1$, $acc' \in E(G)$.
\end{lemma}
\begin{proof}
Suppose for contradiction we had a triple $acc' \notin E(G)$ with $c,c' \in V_3$
and $a \in V_1$. Consider $\Gamma(a,c)$. We know from
Lemmas~\ref{nobipartitebad} that this is a subset of $V_3\cup
V_2\setminus\{c,c'\}$, and must have size at least $m$. Since $\vert
V_3\setminus\{c,c'\} \vert =m-2$, it follows that there must be at least two
vertices $b,b' \in \Gamma(a,c) \cap V_2$.

Now we know that for all $c'' \in V_3$, $bb'c'' \in E(G)$. In particular, for
all $c'' \in V_3\setminus\{c,c'\}$, the triple $acc''$ must also be missing from $E(G)$, since
otherwise we would have $ac \vert bb'c''$. Running through the argument again
with $c''$ instead of $c'$, it follows that $axy$ is missing for \emph{all}
possible choices of distinct $x,y\in V_3$. But then $a \in V_1$ has missing
degree $d_M(a)\geq \binom{m}{2}= \Omega(n^2)$, contradicting
Lemma~\ref{nobigmissingdegree}. Thus all triples taking two vertices in $V_3$
and one vertex in $V_1$ must be in $G$.
\end{proof}

Now let $F$ be the set of tripartite $3$-edges of $G$ associated with the
tripartition $V_1 \sqcup V_2 \sqcup V_3$.
\begin{lemma}\label{n2nooverused}
$F$ contains no overused pairs.
\end{lemma}
\begin{proof}
We consider each possible type of overused pairs in turn, and show they cannot
occur in $G$.
\begin{enumerate}[(i)]
\item Suppose first of all that we had an overused pair $ac$ with $a \in V_1$,
$c \in V_3$. Then there exist $b,b' \in V_2$ such that $abc$ and $ab'c$ are both
in $G$. But then let $c'$ be any element of $V_3 \setminus\{c\}$. We know that
both of $acc'$, $bb'c'$ are in $G$ (by Lemma~\ref{n2mod3nomissingbipartite} and
the preceding remark), so we have $ac \vert bb'c'$, a contradiction.

\item Now suppose that we had an overused pair $bc$ with $b \in V_2$, $c \in
V_3$. Then there exist $a,a' \in V_1$ with $abc, a'bc \in E(G)$. But we know
that for any $b' \in V_2\setminus\{b\}$ we have $bb'c \in E(G)$. In particular
we cannot have $aa'b' \in E(G)$ since otherwise $bc\vert aa'b'$. But we know
that $\Gamma(a,a')\subseteq V_2$ (Lemmas~\ref{nointernalbad} and~\ref{nobipartitebad}), so this would imply that $a,a'$ have codegree at most
$1$, contradicting our minimum codegree condition (provided $n\geq 8$).

\item Finally suppose that we had an overused pair $ab$ with $a \in V_1$ and $b
\in V_2$. Then there exist $c,c' \in V_3$ such that $abc, abc' \in E(G)$. For
any $a' \in V_1\setminus\{a\}$, we have $a'cc' \in E(G)$ (by
Lemma~\ref{n2mod3nomissingbipartite}). In particular we must have $aa'b \notin
E(G)$, since otherwise $ab \vert a'cc'$.

It then follows from our codegree assumption that $\Gamma(a,b)= V_3$. Also, for
all $a' \in V_1 \setminus\{a\}$, $\Gamma(a,a')\subseteq V_2\setminus\{b\}$. By
our codegree assumption again we deduce that $\vert V_2 \vert \geq m+1$, and
hence $\vert V_1 \vert \leq m+1$.

Now for all $a' \in V_1\setminus\{a\}$, we have $\Gamma(a',b)\subseteq \left(V_1
\setminus\{a,a'\}\right)\cup V_3$, so that by the codegree assumption again
there is at least one $c'' \in V_3$ such that $a'bc'' \in E(G)$. The pair $bc''$ is
then an overused pair (used by $a,a'$) taking one vertex in each of $V_2$ and $V_3$,
contradicting (ii). 
\end{enumerate}
\end{proof}

\begin{lemma}
$\vert V_1 \vert = \vert V_2 \vert =m+1$.
\end{lemma}
\begin{proof}
We already know that $m \leq \vert V_1 \vert$ and $\vert V_2 \vert \leq m+2$.
Suppose for contradiction that $\vert V_2\vert =m+2$ and thus $\vert V_1 \vert
=m$. For every $(a,b)\in V_1\times V_2$, we know $\Gamma(a,b)\subseteq
\left(V_1\setminus \{a\}\right) \cup V_3$. Since $\vert V_1
\setminus\{a\}\vert=m-1$, there must be at least one tripartite $3$-edge
containing the pair $(a,b)$. Thus there must be in total at least $\vert V_1
\vert \cdot \vert V_2 \vert = m(m+2)$ distinct tripartite $3$-edges. Averaging
over the $m^2$ pairs $(a,c)\in V_1 \times V_3$, we deduce that at least one such
pair must be contained in at least two tripartite $3$-edges, contradicting
Lemma~\ref{n2nooverused}.

By symmetry, it also cannot be the case that $\vert V_1 \vert =m+2$ and $\vert
V_2 \vert = \vert V_3 \vert =m$, and we are done.
\end{proof}

For every $a,c \in V_1 \times V_3$, we have
$\Gamma(a,c)\subseteq V_2 \cup \left(V_3\setminus\{c\}\right)$. Since
$\delta_2(G) = m$ and $\vert V_3 \vert =m$, it follows that there is at least
one $b\in V_2$ such that $abc \in E(G)$. Furthermore we know this $b$
is unique since the set of tripartite $3$-edges of $G$ contains no overused
pair. Define $\phi(a,c)=b$. 

Also, $\phi^{-1}(b)$ consists of vertex-disjoint pairs (again, as there
are no overused pairs). Thus $\phi$ corresponds to some proper
$(m+1)$-edge colouring of $V_1\times V_3$.
It is easy to see that any $(m+1)$-edge colouring of the
complete bipartite graph $K_{m+1,m}$ extends to that of $K_{m+1,m+1}$
(in fact, in the unique way).
We conclude that $G$ is a subgraph of some $3$-graph in $\CT(n+1)$
and thus of some
$H\in\CT(n)$. This finishes the proof of Theorem~\ref{n=2mod3config}.\end{proof}

\begin{remark}\label{rm:2mod3} Note that an extremal $G$ with $\vert V_3 \vert=\frac{n-2}{3}$ can have some edges of the form $aa'b$
with $a,a'\in V_1$ and $b\in V_2$ missing. Namely, if there
exist $c,c' \in V_3$ such that $abc$ and $a'bc'$ are both $3$-edges of $G$, then
we may delete $aa'b$ without lowering the codegree of $G$. On the other hand,
for
each pair $a,a'\in V_1$ we have at most one $b\in
V_2$ for which $aa'b$ is missing, and similarly for every pair $(a,b) \in V_1
\times V_2$ we have at most one $a'$ for which $aa'b$ is missing.
\end{remark}

\subsubsection{The case $n$ congruent to $1$ modulo $3$}

In this section, let $n=3m+1$ be congruent to $1$ modulo $3$ and sufficiently large. Unlike the two previous cases, the upper bound in Corollary~\ref{upperboundcoex} is not
sharp.

\begin{proposition}\label{thresholdn1mod3}
For all $n$ congruent to $1$ modulo $3$ and sufficiently large,
$\mathrm{coex}(n, F_{3,2})=\lfloor n/3\rfloor-1$.
\end{proposition}

\begin{proof}
Let $n=3m+1$ be large, and $G$, $V_1$, $V_2$, $V_3$ be as in Section~\ref{stabilityanalysis}. Suppose for contradiction that $\delta_2(G)=m$. Consider any pair of vertices from $V_1$. By Lemmas~\ref{nointernalbad} and~\ref{nobipartitebad}, their joint neighbourhood is a subset of $V_2$, so that by the codegree condition we must have $\vert V_2\vert \geq m$. Similarly we have $\vert V_3\vert$ and $\vert V_1 \vert$ both at least $m$, so that in fact we must have two parts of size $m$ and one part of size $m+1$. Assume without loss of generality that $\vert V_3 \vert = m+1$, and that $\vert V_1\vert = \vert V_2 \vert=m$.

By the codegree condition, all edges with two vertices in $V_3$ and one in $V_1$
or two vertices in $V_1$ and one vertex in $V_2$ must be in $E(G)$. In addition,
for every pair $(b,c) \in V_2 \times V_3$, we know that $\Gamma(b,c)\subseteq V_1
\cup \left(V_2 \setminus \{b\}\right)$. Since $(b,c)$ has codegree 
at least $m$ and
$\vert V_2 \vert =m$, it follows that there exists at least one $a\in V_1$ such
that $abc \in E(G)$. Summing over all possible pairs $(b,c)$, we see that there
must be at least $m(m+1)$ tripartite $3$-edges in $G$. But there are only $m^2$
distinct pairs $(a,b) \in V_1 \times V_2$. Thus there is at least one such pair
appearing in at least two tripartite $3$-edges, 
i.e.\ there must be $a \in V_1$, $b \in V_2$, $c,c' \in V_3$ such that both
$abc$ and $abc'$ are in $E(G)$.

But then let $a'$ be any vertex in $V_1 \setminus\{a\}$. By our earlier observations, we know that $aa'b$ and $cc'a'$ are both $3$-edges of $G$, so that $ab \vert cc'a'$, contradicting the fact that $G$ is $F_{3,2}$-free.
\end{proof}

A consequence of this lower codegree threshold is that the extremal structures are considerably more complicated. We present three families $\CT_1(n)$, $\CT_2(n)$ and $\CT_3(n)$ of extremal 
3-graphs on $[n]$ and show
that for every extremal $G$ there is some $H\in\cup_{i=1}^3 \CT_i(n)$ 
containing $G$ as a (spanning) subgraph. One could say more about the possible structure of $E(H)\setminus E(G)$ (along the lines of
Remark~\ref{rm:2mod3}) but we do not think that this description will
be very illuminating. Let us define each family $\CT_i(n)$.

\begin{construction}[Family $\CT_1(3m+1)$]\label{co:H1}
Start with $T_{A,B,C}$ where $|A|=m$, $|B|=m+2$ and $|C|=m-1$. Add an arbitrary
set of tripartite edges so that no overused pairs are created and for
every $a\in A$ and $c\in C$ there is a tripartite edge containing $\{a,c\}$.
\end{construction}

\begin{construction}[Family $\CT_2(3m+1)$]\label{co:H2}
 Let $0\le k\le m+1$. Start with $T_{A,B,C}$ where
$|A|=|B|=m+1$ and $|C|=m-1$. Let $S$ consist of $k$ vertex-disjoint
pairs from $A\times B$. 

Remove all $3$-edges of $T_{A,B,C}$ that contain a pair from $S$.
Add all tripartite $3$-edges that contain a pair from $S$. Thus 
$S$ is precisely the set of overused pairs now. Add an
arbitrary collection of tripartite $3$-edges so that no new overused
pair is created and for every $a\in A$ and $c\in C$ there is at least
one tripartite edge containing $\{a,c\}$. (Note that if $a$ belongs
to a pair in $S$, then this condition is automatically satisfied.)
\end{construction}

\begin{construction}[Family $\CT_3(3m+1)$]\label{co:H3}
Start with $T_{V_1,V_2,V_3}$, where $\vert V_1 \vert =m+1$ and $\vert V_2 \vert
= \vert V_3 \vert =m$. 

Let $S$ consist of pairs of
vertices, containing at most one pair from $V_i\times V_{i+1}$ for each $i\in
[3]$ so that if $i\in \{1,3\}$  and $S$ contains both $(x,y)\in V_{i-1}\times
V_{i}$ and $(y',z)\in V_{i}\times V_{i+1}$, then $y=y'$. 
(Thus $0\le |S|\le 3$; for example, if $|S|=3$ then the pairs in $S$ form
either a 3-cycle or a path ending and starting in $V_2$.)

Remove all 3-edges from $T_{V_1,V_2,V_3}$ that contain a  pair in $S$.
Add an arbitrary collection of tripartite
3-edges so that 
 \begin{itemize}
 \item each pair of $S$ is contained in at least $m-1$ added edges;
 \item there are no overused pairs other than those from $S$;
 \item if $|V_i|=m$ (that is, $i\in \{2,3\}$) and $(x,y)\in V_i\times V_{i+1}$
is in $S$, then for every $x'\in V_i\setminus\{x\}$ the pair $\{x',y\}$ is
contained in exactly
one tripartite edge.
 \end{itemize} 
\end{construction}

We leave it to the reader to verify that each constructed $3$-graph
has minimum codegree $m-1$. The following result implies that all
these $3$-graphs are $F_{3,2}$-free.

\begin{proposition}\label{weaksuitedmatchingsaregood}
Let $V$ be a set of vertices with tripartition $V=V_1 \sqcup V_2 \sqcup V_3$. Let $G$ be obtained from $T_{V_1,V_2,V_3}$ by adding
some set $F$ of tripartite $3$-edges and removing all $3$-edges of $T_{V_1,V_2,V_3}$ that contain a pair overused by $F$. Then $G$ is $F_{3,2}$-free.
\end{proposition}
\begin{proof}
 By
Proposition~\ref{matchingsaregood} we need only to check for copies
of $F_{3,2}$ that contain two tripartite edges sharing an overused pair, 
say $abc, ab'c \in F$ with $a \in V_1$, $c \in V_3$
and $b,b' \in V_2$. Each such $F_{3,2}$ has to be of form $ac
\vert bb'x$  for some $x$. Now, $bb'x \in E(G)$
implies $x \in V_3$. Since $(a,c)$ is an overused
pair, we have $acx \notin E(G)$ by the definition of $G$. Thus
we cannot have $ac \vert bb'x$, as desired.
\end{proof}

Examples of $3$-graphs in $\CT_1(n)$, $\CT_2(n)$ and $\CT_3(n)$ can be obtained by 
taking a $3$-graph in respectively  $\CT(n+5)$, $\CT(n+2)$ and $\CT(n+2)$, and
deleting
arbitrary vertices so that the parts have the desired sizes. However,
note that, for example, not all 3-graphs in $\CT_2(n)\cup \CT_3(n)$ 
with $S=\emptyset$ come from
$\CT(n+2)$ as there are $(m+1)$-edge colourings 
of $K_{m+1,m-1}$ (for $m\geq 4$) and $K_{m,m}$ (for $m \geq 2$) that do not extend to an 
$(m+1)$-edge colouring of $K_{m+1,m+1}$.

We shall show that the $3$-graphs in $\cup_{i=1}^3 \CT_i(n)$ contain (as
spanning subgraphs)
all possible extremal configurations of order $n$. We know
from our analysis in Section~\ref{stabilityanalysis} that every extremal
configuration $G$ for the codegree problem
consist of subgraph of $T_{V_1,V_2,V_3}$ together with a set of
tripartite $3$-edges. Thus the minimum codegree is at most $\min(\vert V_i
\vert:i \in [3])$. As 
$\delta_2(G)= m-1$, we must have $\vert V_i \vert \ge
m-1$ for every $i \in [3]$. We separate out into two cases according to whether
or not we have equality for some $i$.

\begin{theorem}\label{n=1mod3configt2}
Let $G$, $V_1$, $V_2$, $V_3$ be as in Section~\ref{stabilityanalysis}, and
suppose $n=3m+1$ is large and $\delta_2(G) =m-1$. If $\vert V_i \vert =m-1$ for
any $i=1,2,3$, then $G$ is isomorphic to a subgraph of some $H\in
\CT_1(n)\cup \CT_2(n)$.
\end{theorem}
\begin{proof}
Without loss of
generality, assume that $|V_3|=m-1$. By Lemmas~\ref{nointernalbad}
and~\ref{nobipartitebad}, we have that $\Gamma(x,x')\subseteq
V_3$ for every $x,x' \in V_{2}$. The codegree condition $\delta_2(G)\geq m-1$
then implies that all
$3$-edges taking two vertices in $V_{2}$ and one in $V_3$ are in $G$. In
addition, we have:

\begin{lemma}\label{v3v3v1allin}
All $3$-edges taking two vertices in $V_{3}$ and one in $V_{1}$ are in $G$.
\end{lemma}
\begin{proof}
Indeed, suppose that $acc'\notin E(G)$ for some $c,c' \in V_3$ and $a\in V_1$. 
Since $\Gamma(c,a)$ contains at least $m-1$ vertices and is contained in
$V_2\cup V_3\setminus\{c,c'\}$ and since $V_3 \setminus\{c,c'\}$ has size
$m-3$, 
it follows that there exist $b,b'$ such that $abc$ and $ab'c$ are both in
$E(G)$. But then for all $x\in V_3\setminus\{c\}$, the $3$-edge $acx$ cannot be in $G$,
for otherwise $ac\vert bb'x$. Likewise, for every $y \in V_3\setminus\{x\}$
we have that $axy$ is missing from $G$. This implies $d_M(a)\geq
\binom{m-1}{2}=\Omega(n^2)$, contradicting Lemma~\ref{nobigmissingdegree}.
\end{proof}

With Lemma~\ref{v3v3v1allin} in hand, we can now turn our attention to the tripartite $3$-edges of $G$. Write $F$ for the tripartite $3$-edges associated with the tripartition $V_1 \sqcup V_2 \sqcup V_3$.
\begin{corollary}\label{case1overusedv1v3}
$V_1 \times V_3$ contains no overused pair.
\end{corollary}
\begin{proof}
Suppose we had $a \in V_1$, $b,b' \in V_2$ and $c \in V_3$ with $abc, ab'c \in F$. Then for all $c' \in V_3\setminus\{c\}$ we must have $acc'$ missing from $G$ to prevent $ac\vert bb'c'$, contradicting Lemma~\ref{v3v3v1allin} (recall that $bb'c \in E(G)$, as observed just before Lemma~\ref{v3v3v1allin}). 
\end{proof}

Next we show that $V_2 \times V_3$ does not contain overused pairs either.
\begin{lemma}\label{case1overusedv2v3}
	$V_2 \times V_3$ contains no overused pairs 
\end{lemma}
\begin{proof}
	Suppose we had $a,a' \in V_1$, $b \in V_2$ and $c \in V_3$ such that $abc$ and $a'bc$ are both in $F$. We know that $\Gamma(a,a') \subseteq V_2$ (by Lemmas~\ref{nointernalbad} and~\ref{nobipartitebad}), so provided $n$ is sufficiently large (which we are assuming) there is at least one $b' \in V_2 \setminus\{b\}$ such that $aa'b' \in E(G)$. But since we also have $bb'c \in E(G)$ (as observed just before Lemma~\ref{v3v3v1allin}), this means $bc\vert aa'b'$, a contradiction.
\end{proof}

In particular, all overused pairs from $F$ come from $V_1 \times V_2$. 
\begin{lemma}\label{case1propoverusedv1v2}
	Let $(a,b)\in V_1 \times V_2$ be an overused pair from $F$. Then the following hold:
	\begin{enumerate}[(i)]
		\item $\Gamma(a,b) =V_3$;
		\item $\{f\in F: \ a \in f\}=\{f \in F: \ b \in f\}$.
	\end{enumerate}
\end{lemma}
\begin{proof}
	Let $(a,b)\in V_1\times V_2$ be such an overused pair. Then there exist	$c,c' \in V_3$ such that $abc$ and $abc'$ are $3$-edges of $G$.

	By Lemma~\ref{nobipartitebad}, we know $\Gamma(a,b)\subseteq V_1\cup V_3$. Suppose $aa'b\in E(G)$ for some $a' \in V_1$. By Lemma~\ref{v3v3v1allin}, we know $a'cc'\in E(G)$, so that $ab\vert a'cc'$, a contradiction. Thus $\Gamma(a,b)\subseteq V_3$, and the codegree condition $d(a,b)\geq m-1=\vert V_3\vert$ tells us $\Gamma(a,b)=V_3$, proving Part (i) of the lemma.

	Part (ii) is then immediate from Corollary~\ref{case1overusedv1v3} and Lemma~\ref{case1overusedv2v3}: if $ab'c''\in E(G)$ for some $b' \in V_2\setminus\{b\}$ and $c''\in V_3$, then $(a,c'')$ is an overused pair (used by $b$ and $b'$) from $V_1 \times V_3$, contradicting Corollary~\ref{case1overusedv1v3}; similarly if $a'bc''\in E(G)$ for some $a'\in V_1\setminus\{a\}$ and $c'' \in V_3$, then $(b,c'')$ is an overused pair (used by $a$ and $a'$) from $V_2\times V_3$, contradicting Lemma~\ref{case1overusedv2v3}.\qedhere
		
\end{proof}
Note Lemma~\ref{case1propoverusedv1v2} implies that the overused pairs from $F$ are vertex-disjoint pairs from $V_1\times V_2$.

For every pair $(a,c) \in V_1 \times V_3$, the joint neighbourhood $\Gamma(a,c)$
is a subset of $V_2 \cup \left(V_3 \setminus\{c\}\right)$. 
By the codegree condition $\delta_2(G) \geq m-1$ and the fact that $\vert V_3 \vert =m-1$, it follows that for every such pair there is at least one tripartite $3$-edge $abc\in F$ with $b \in V_2$. Now there are exactly $(m-1)\vert V_1 \vert$ distinct such pairs $(a,c) \in V_1 \times V_3$. On the other hand, since there are no overused $V_2 \times V_3$ pairs arising from $F$, there can be at most $(m-1)\vert V_2 \vert$ such tripartite $3$-edges, one for each pair $(b,c)\in V_2 \times V_3$. 
Thus $\vert V_2 \vert \geq \vert V_1 \vert$.

If $\vert V_2 \vert = \vert V_1 \vert =m+1$, then by adding all missing $V_1V_1V_2$
$3$-edges to $G$ we obtain a member of $\CT_2(n)$, as
desired.

So let us suppose that $\vert V_1 \vert \leq m$. We know from our
codegree condition that $\vert V_1 \vert \geq m-1$, and the inequality $\vert
V_1 \vert \leq m$ implies $\vert V_2 \vert \geq m+2$.

We claim that $F$ contains no overused pair. Indeed, suppose $(a,b)\in V_1 \times V_2$ is an overused pair. By Lemma~\ref{case1propoverusedv1v2} Part (i), $aa'b\notin E(G)$ for all $a' \in V_1\setminus\{a\}$. For each $a' \in V_1\setminus\{a\}$, the codegree condition then tells us that $\Gamma(a',b)$ is a subset of $\left(V_1 \setminus\{a,a'\}\right)\cup V_3$ of
size at least $m-1$. In particular there must exist $c \in V_3$ with $a'bc \in
E(G)$. But this is a tripartite $3$-edge containing $b$ and not $a$, contradicting Part (ii) of Lemma~\ref{case1propoverusedv1v2}. Thus $F$ has no overused pair, as claimed.

Next, suppose that $\vert V_1 \vert =m-1$. Then for every
$(a,b) \in V_1 \times V_2$, $\Gamma(a,b)\subseteq \left(V_1
\setminus\{a\}\right) \cup V_3$. By the codegree assumption $\delta_2(G) \geq
m-1$, we deduce that there must be at least one tripartite $3$-edge involving
the pair $(a,b)$. Thus there must be at least $\vert V_1 \vert \cdot \vert V_2
\vert > \vert V_1 \vert \cdot \vert V_3 \vert$ tripartite $3$-edges in $G$,
implying the existence of an overused pair in $V_1 \times V_3$,
contradicting Corollary~\ref{case1overusedv1v3}. Thus $\vert V_1 \vert =m$, and hence $\vert V_2\vert=m+2$.

As observed after Lemma~\ref{case1propoverusedv1v2} above, every pair $(a,c)\in V_1\times V_3$ is covered 
by at least one tripartite $3$-edge (otherwise its codegree is at most
$\vert V_3\vert-1<m-1$); we have already shown that there are no overused pairs in $F$. By adding all missing $3$-edges of the form $V_1V_1V_2$ to
$G$ we thus obtain a member of $\CT_1(n)$, as
required.\end{proof}

\begin{theorem}\label{n=1mod3configt1}
Let $G$, $V_1$, $V_2$, $V_3$ be as in Section~\ref{stabilityanalysis}, and
suppose $n=3m+1$ is large and $\delta_2(G) =m-1$. If $\vert V_i \vert \geq m $
for all $i\in [3]$, 
then $G$ is a subgraph of some $H\in \CT_3(n)$.
\end{theorem}
\begin{proof}
Assume without
loss of generality that $\vert V_1 \vert = m+1$ and $\vert V_2 \vert = \vert V_3
\vert = m$.

Let us show first that overused pairs are contained in tripartite
$3$-edges only.
\begin{lemma}\label{overused in tripartite only}
If $(x,y)$ is an overused pair in $V_i \times V_{i+1}$, then $\Gamma (x,y)\subseteq V_{i-1}$. 
\end{lemma}
\begin{proof}
Since $(x,y)$ is an overused pair, there exist $z,z'$ in $V_{i-1}$ such that $xyz, xyz'$ are $3$-edges of $G$. Now $\Gamma(z,z')\subseteq V_i$ (by Lemmas~\ref{nointernalbad} and~\ref{nobipartitebad}) so that by the codegree condition $\Gamma(z,z')$ contains at least $m-2$ elements of $\vert V_i \setminus\{x\}\vert$. For any such element $x'$, $xx'y\notin E(G)$ for otherwise we would have $xy \vert x'zz'$. Now the joint neighbourhood of $x$ and $y$ is contained in $V_i \cup V_{i-1}$ (Lemma~\ref{nobipartitebad}) and has size at least $m-1$, from which it follows that 
\begin{align*}
\vert \Gamma(x,y)\cap V_{i-1}\vert &\geq m-1- \left(\vert V_i\setminus\{x\}\vert - (m-2)\right)\\
&= 2m-3 - \vert V_i \setminus\{x\}\vert\\
&\geq m-3.
\end{align*}
Now suppose $xx'y \in E(G)$ for some $x' \in V_i$. Then for all $w,w' \in \Gamma(x,y)\cap V_{i-1}$ we would have $x'ww' \notin E(G)$, for otherwise $xy\vert x' ww'$. But then $d_M(x')\geq \binom{m-3}{2}=\Omega(n^2)$, contradicting Lemma~\ref{nobigmissingdegree}. Thus if $(x,y)$ is an overused pair from $V_i \times V_{i+1}$ then $\Gamma(x,y)\subseteq V_{i-1}$.
\end{proof}

We now turn our attention to showing that for each $i\in\{1,2,3\}$, the set $V_i \times V_{i+1}$ contains at most one overused pair.

\begin{lemma}\label{caset1vi+1is}
If $\vert V_{i+1}\vert =m$ and $(a,b)$, $(a',b')$ are overused pairs from $V_i \times V_{i+1}$, then $b=b'$.
\end{lemma}
\begin{proof}
Suppose not. We know by Lemma~\ref{overused in tripartite only} that for all $a'' \in V_i$, neither of $aa''b$ and $a'a''b'$ are $3$-edges of $G$.

If $a=a'$, then we have for any $a'' \in V_i \setminus\{a\}$ that
\[ \vert \Gamma(a,a'')\vert \leq \vert V_{i+1} \setminus \{b,b'\}\vert=m-2,\]
contradicting our codegree assumption $\delta_2(G)= m-1$. On the other hand, if
$a\neq a'$ then
\[\vert \Gamma(a,a')\vert \leq \vert V_{i+1} \setminus \{b,b'\}\vert=m-2,\]
contradicting again the codegree assumption.
\end{proof}
\begin{lemma}\label{vivi+1noaa'b}
Suppose $(a,b)$ and $(a',b)$ are overused pairs from $V_i \times V_{i+1}$. Then $a=a'$.
\end{lemma}
\begin{proof}
By Lemma~\ref{overused in tripartite only}, we know that $\Gamma(a,b)$ and $\Gamma(a',b)$ are both subsets of $V_{i-1}$ of size at least $m-1$. In particular since $\vert V_{i-1}\vert \leq m+1$, we have that
$\Gamma(a,b)\cap \Gamma(a',b)$ is a subset of $V_{i-1}$ of size at least $m-3$.

Now we know from Lemma~\ref{nobigmissingdegree} that $d_M(b)=o(n^2)=o(m^2)$. Thus for all but $o(m)$ vertices $b' \in V_{i+1} \setminus\{b\}$, we have that $bb'c \in E(G)$ for all but $o(m)$ vertices $c \in \Gamma(a,b)\cap \Gamma(a,b')$.

But for such $b'$ and $c$, $aa'b' \notin E(G)$, for otherwise we would have $bc \vert aa'b'$. Thus $\Gamma(a,a')$ (which we know is a subset of $V_{i+1}$) can contain at most $o(m)$ vertices, contradicting our codegree assumption for $n$ (and hence $m$) sufficiently large.
\end{proof}

Taken together, the last two lemmas imply the following:
\begin{corollary}\label{v1v2 and v2v3 at most one}
$V_1 \times V_2$ and $V_2 \times V_3$ each contain at most one overused
pair.\qed
\end{corollary}

We now prove analogues of Lemma~\ref{caset1vi+1is} for $V_3 \times V_1$, to show that it also contains at most one overused pair.

\begin{lemma}\label{v1v3noaa'c}
Suppose $(c,a)$ and $(c,a')$ are overused pairs from $V_3 \times V_1$. Then $a=a'$.
\end{lemma}
\begin{proof}
Suppose not. Then by Lemma~\ref{overused in tripartite only} we know that $\Gamma(a,c)$ and $\Gamma(a',c)$ are subsets of $V_2$ of size at least $\delta_2(G)=m-1$. We also know (Lemmas~\ref{nointernalbad} and~\ref{nobipartitebad}) that $\Gamma(a,a')$ is a subset of $V_2$ of size at least $\delta_2(G)=m-1$. Thus the intersection
\[ I=\Gamma(a,c)\cap \Gamma(a',c) \cap \Gamma(a,a')\]
has size at least $3(m-1)-2\vert V_2\vert= m-3$.

For every distinct $b,b'\in I$ we have that $bb'c\not\in E(G)$ because
otherwise we have $bc|aa'b'$. But then $d_M(c)\ge {|I|\choose 2}$,
contradicting Lemma~\ref{nobigmissingdegree}.\end{proof}

\begin{lemma}\label{v1v3 only one overused}
Suppose $(c,a)$ and $(c',a')$ are overused pairs from $V_3 \times V_1$. Then $a=a'$ and $c=c'$. (In particular, $V_1 \times V_3$ contains at most one overused pair.)
\end{lemma}
\begin{proof}
Suppose not. The only case left over from Lemmas~\ref{vivi+1noaa'b}
and~\ref{v1v3noaa'c} is the case when both $a\neq a'$ and $c\neq c'$, 
i.e.\ when we have vertex-disjoint overused pairs.

By Lemma~\ref{overused in tripartite only}, we know that $\Gamma(a,c)$ and $\Gamma(a',c')$ are both subsets of $V_2$. Now consider an arbitrary $c'' \in V_3 \setminus \{c,c'\}$. Since $acc'' \notin E(G)$ and $\vert V_3 \setminus\{c,c''\}\vert =m-2$, there must exist $b=b(c'')\in V_2$ such that $abc'' \in E(G)$. Similarly there must exist $b'=b'(c'') \in V_2$ such that $a'b'c'' \in E(G)$.

Now note that if $b\in \Gamma(a,c)$ then $(a,b)$ is overused (since both $abc$
and $abc''$ are in $G$). Similarly, if $b' \in \Gamma(a',c')$ then $(a',b')$ is
overused.

Also, $V_2$ has size $m$ while $\Gamma(a,c)$ and $\Gamma(a',c')$ both have size at least $m-1$. So there is at most one vertex $b_{\star} \in V_2 \setminus \Gamma(a,c)$ and at most one vertex $b'_{\star} \in V_2 \setminus \Gamma(a',c')$.

We now apply the pigeonhole principle to get a contradiction for $m$ large enough (at least $4$):
\begin{itemize}
\item if $b(c'')=b_{\star}$ for at least two distinct $c'' \in V_3\setminus\{c,c'\}$ then $(a,b_{\star})$ is as overused pair;
\item if $b(c'') \neq b_{\star}$ for at least one $c''\in V_3\setminus\{c,c'\}$ then $(a,b(c''))$ is an overused pair;
\item if $b'(c'')=b'_{\star}$ for at least two distinct $c'' \in V_3\setminus\{c,c'\}$ then $(a',b'_{\star})$ is an overused pair;
\item if $b'(c'') \neq b'_{\star}$ for at least one $c''\in V_3\setminus\{c,c'\}$ then $(a',b'(c''))$ is an overused pair.
\end{itemize}
Thus provided $\vert V_3 \setminus\{c,c'\}\vert \geq 2$, we have at least two distinct overused pairs from $V_1 \times V_2$, one involving $a$ and the other  $a'$. This contradicts Corollary~\ref{v1v2 and v2v3 at most one}.
\end{proof}

We have thus shown that for every $i \in [3]$, $V_i \times V_{i+1}$ contains at
most one overused pair. 

\begin{lemma}\label{lm:op1} If $(x,y)\in V_i\times V_{i+1}$ is an overused
pair
and $|V_{i}|=m$, then for every $x'\in V_i\setminus\{x\}$ there is exactly one
$z\in V_{i-1}$ with $\{x',y,z\}\in E(G)$.
\end{lemma}
\begin{proof}
 The joint neighbourhood of $x',y$ lies inside $V_{i-1}\cup
V_i\setminus\{x,x'\}$. Since $\delta_2(G)\ge m-1$, there must exists at least
one $z$ as required. Since $\{x',y\}$ is not an overused pair, this $z$ is
unique.\end{proof}

\begin{lemma}\label{no ac-bc bad cnotequal}
Suppose $(a,c)$ and $(b', c')$ are overused pairs from $V_1 \times V_3$ and $V_2 \times V_3$ respectively. Then $c=c'$.
\end{lemma}
\begin{proof}
Suppose not. For $b''\in V_2\setminus\{b'\}$ let
$z(b'')$ be the vertex in $V_1$ with $\{b'',c',z(b'')\}\in E(G)$ given
by Lemma~\ref{lm:op1}.

If $a'=z(b_1'')=z(b_2'')$ for some distinct $b_1'', b_2'' \in V_2
\setminus\{b'\}$, then we have that $(a',c')$ is an overused pair from $V_1
\times V_3$ distinct from $(a,c)$ (since $c\neq c'$), contradicting 
Lemma~\ref{v1v3 only one overused}. Thus the map $z:V_2\setminus\{b_1\}\to V_1$
is injective. 

By Lemma~\ref{overused in tripartite only}, $\Gamma(b',c')$ is a
subset of $V_1$ of size at least $m-1$. As $n$ is large, $\Gamma(b',c')$ must
contain some $a'=z(b'')$. But then $a'c'b',a'c'b''\in E(G)$ so $a'c'$ is an
overused pair from $V_1 \times V_3$ distinct from $(a,c)$ (since $c\neq c'$),
again contradicting Lemma~\ref{v1v3 only one overused}.\end{proof}

Similarly, we have
\begin{lemma}\label{no ac-ab bad anotequal}
Suppose $(a,c)$ and $(a', b')$ are overused pairs from $V_1 \times V_3$ and $V_1 \times V_2$ respectively. Then $a=a'$.
\end{lemma}
\begin{proof}
Identical to the proof of Lemma~\ref{no ac-bc bad cnotequal}, with $V_i$ playing the role of $V_{i-1}$. 
\end{proof}

The above lemmas show that if we add all edges from $T_{V_1,V_2,V_3}$ to
$G$, we obtain an element of $\CT_3(n)$, as claimed.
\end{proof}

\section{Tur\'an density subject to a codegree constraint}
A natural variation of the Tur\'an density and codegree density problems is the following. 
\begin{definition}
Let $\mathcal{F}$ be a family of nonempty $3$-graphs, and let $(c_n)_{n \in \mathbb{N}}$ be a sequence of real numbers with $c_n \in [0,\frac{\mathrm{coex}(n, \mathcal{F})}{n-2}]$ for each $n\in \mathbb{N}$. The \emph{Tur\'an number of $\mathcal{F}$ subject to the codegree constraint $(c_n)_{n\in \mathbb{N}}$} is the function $\mathrm{ex}_{c_n}(\cdot, \mathcal{F})$ sending $n \in \mathbb{N}$ to the maximum number of $3$-edges in an $\mathcal{F}$-free $n$-vertex $3$-graph with minimum codegree at least $c_n(n-2)$.
\end{definition}

\begin{problem}\label{Turanwithcodeg}
Let $\mathcal{F}$ be a family of nonempty $3$-graphs, and let $c \in [0, \gamma(\mathcal{F}))$. Determine $\mathrm{ex}_{c}(n, \mathcal{F})$.
\end{problem}
To the best of our knowledge, Lo and Markstr\"om~\cite{LoMarkstrom12} were 
the first to pose a question of the kind considered in
Problem~\ref{Turanwithcodeg}. They asked for the behaviour of
$\textrm{ex}_{c}(n, \mathcal{F})$ when $\mathcal{F}$ is the $3$-graph
$K_4^-$. 

Problem~\ref{Turanwithcodeg} can be thought of as a way of viewing
Problems~\ref{Turannumber} and~\ref{codegreethreshold} together within a common
framework. In addition codegree constraints are natural in the context of
$3$-graphs, so that Problem~\ref{Turanwithcodeg} is appealing from an extremal
hypergraph perspective.

For the Fano plane $F_7$, Problem~\ref{Turanwithcodeg} is trivial from the work of Keevash and Sudakov~\cite{KeevashSudakov05}, F\"uredi and Simonovits~\cite{FurediSimonovits05} and Keevash~\cite{Keevash09}: the extremal configurations for the Tur\'an number and for the codegree threshold are identical for all $n$ sufficiently large, so that $\mathrm{ex}_{c}(n, F_7)= \mathrm{ex}(n, F_7)$ for all $c\in[0, 1/2]$ and all but finitely many $n$.

The situation is very different for $F_{3,2}$, where codegree-extremal
configurations have $n^3/18+o(n^3)$ $3$-edges, as we have shown, while the
extremal configurations have $2n^3/27+o(n^3)$ $3$-edges, 
i.e.\ about one and a third times as many. A first step towards the resolution of
Problem~\ref{Turanwithcodeg} for $F_{3,2}$ would be to identify the asymptotic
behaviour of  $\mathrm{ex}_{c}(n, F_{3,2})$ for $c\in[0,1/3]$.

A lower bound can be obtained by shifting weight in a continuous fashion from part $A$ to part $C$ in a $T_{A,B,C}$ construction, and so to move from Construction~\ref{onewaybip} (where $\vert A\vert =\frac{2n}{3}+O(1)$, $\vert B\vert =\frac{n}{3}+O(1)$ and $\vert C \vert =0$) to Construction~\ref{orcyclebip} (where all three parts have size $\frac{n}{3}+O(1)$). For $c\in[0,1/3]$, this gives the following:
\[\textrm{ex}_{c}(n, F_{3,2})\geq \left(\frac{1}{3}+3\left(\frac{1}{3}-c\right)^3\right)\binom{n}{3}+o(n^3).\]
\begin{question}
	Is this lower bound asymptotically best possible?
\end{question}

\subsection*{Acknowledgements} We are grateful to an anonymous referee for a careful reading of a long paper.

\bibliographystyle{siam}
\bibliography{codegreebiblio}

\end{document}